\documentclass[10pt,dvipsnames]{amsart}

\usepackage[utf8]{inputenc}
\usepackage{lmodern}
\usepackage[T1]{fontenc}
\usepackage[english]{babel}
\usepackage{fullpage}

\usepackage{latexsym}
\usepackage{amsmath,amssymb,amsthm}
\usepackage{mathrsfs}
\usepackage{stmaryrd}

\usepackage{enumerate}
\usepackage{enumitem}

\usepackage{rotating}

\usepackage{todonotes}
\usepackage{tikz-cd}
\usetikzlibrary{shapes}
\usepackage{tikz}

\makeatletter
\DeclareRobustCommand{\rvdots}{%
	\vbox{
		\baselineskip4\p@\lineskiplimit\z@
		\kern-\p@
		\hbox{.}\hbox{.}\hbox{.}
}}
\makeatother

\usepackage{xcolor}
\usepackage{color}
\usepackage{hyperref}
\hypersetup{
	colorlinks=true,
	linkcolor=blue,
	citecolor=blue,
	urlcolor=blue
}

\usepackage{ulem}
\usepackage{soul}

\newtheorem{theoremAlph}{Theorem}
\newtheorem{corollaryAlph}[theoremAlph]{Corollary}

\newtheorem{theorem}{Theorem}[section]
\newtheorem{lemma}[theorem]{Lemma}	
\newtheorem{proposition}[theorem]{Proposition}
\newtheorem{corollary}[theorem]{Corollary}
\theoremstyle{definition}
\newtheorem{definition}[theorem]{Definition} 
\newtheorem{remark}[theorem]{Remark}	

\newtheorem{question}[theorem]{Question}

\theoremstyle{definition} 
\newtheorem*{ack}{Acknowledgements}

\numberwithin{equation}{section}

\newcommand{\C}{\mathbb{C}} 
\newcommand{\K}{\mathbb{K}} 
\newcommand{\R}{\mathbb{R}} 
\newcommand{\Q}{\mathbb{Q}} 
\newcommand{\Z}{\mathbb{Z}} 
\newcommand{\N}{\mathbb{N}} 
\newcommand{\Quat}{\mathbb{H}}

\newcommand{\II}{\mathrm{I\!I}}
\newcommand{\Ric}{\textup{Ric}}

\newcommand{\ttimes}{\mathrel{\widetilde{\times} }}

\newcommand{\bigslant}[2]{{\raisebox{.2em}{$#1$}\left/\raisebox{-.2em}{$#2$}\right.}}
\newcommand{\cone}[1]{\overset{\scalebox{0.5}[0.3]{$\boldsymbol{<}$}}{#1}}

\makeatletter
\DeclareRobustCommand*\uell{\mathpalette\@uell\relax}
\newcommand*\@uell[2]{
	\setbox0=\hbox{$#1\ell$}
	\setbox1=\hbox{\rotatebox{10}{$#1\ell$}}
	\dimen0=\wd0 \advance\dimen0 by -\wd1 \divide\dimen0 by 2
	\mathord{\lower 0.1ex \hbox{\kern\dimen0\unhbox1\kern\dimen0}}
}

\makeatletter
\newcommand{\mylabel}[2]{#2\def\@currentlabel{#2}\label{#1}}
\makeatother

\begin{document}
	\title[\rmfamily Positive Ricci curvature on connected sums of fibre bundles]{\rmfamily Positive Ricci curvature on connected sums of fibre bundles}
	\date{\today}
	\author{Philipp Reiser}
	\address{Department of Mathematics, University of Fribourg, Switzerland}
	\email{\href{mailto:philipp.reiser@unifr.ch}{philipp.reiser@unifr.ch}}
	
	\normalem
	
	\begin{abstract}
		We consider the problem of preserving positive Ricci curvature along connected sums. In this context, based on earlier work by Perelman, Burdick introduced the notion of \emph{core metrics} and showed that the connected sum of manifolds with core metrics admits a Riemannian metric of positive Ricci curvature. He subsequently showed that core metrics exist on compact rank one symmetric spaces, certain sphere bundles and manifolds obtained as boundaries of certain plumbings. In this article, we show that core metrics can be lifted along general fibre bundles, including sphere bundles and projective bundles. Our techniques also apply to spaces that decompose as the union of two disc bundles such as the Wu manifold. As application we show that all classes in the torsion-free oriented bordism ring can be represented by connected manifolds of positive Ricci curvature.
	\end{abstract}

	\maketitle

	\section{Introduction and Main Results}
	
	This article is motivated by the following question:
	
	\begin{question}\label{Q:conn_sum}
		Let $M_1^n$, $M_2^n$ be closed $n$-manifolds that admit a Riemannian metric of positive Ricci curvature. Does the connected sum $M_1\# M_2$ also admit a Riemannian metric of positive Ricci curvature?
	\end{question}
	
	If both $M_1$ and $M_2$ are non-simply-connected, then the connected sum $M_1\# M_2$ has an infinite fundamental group, which implies by the theorem of Bonnet--Myers that it cannot admit a Riemannian metric of positive Ricci curvature. However, if at least one of $M_1$ and $M_2$ is simply-connected, then Question \ref{Q:conn_sum} is open.
	
	First examples of Riemannian metrics of positive Ricci curvatures on connected sums were given by Cheeger \cite{Ch73} on the connected sum of two symmetric spaces of rank one and later by Sha--Yang \cite{SY89,SY91} on arbitrary connected sums of copies of $S^p\times S^q$ with fixed $p,q\geq 2$. The latter was obtained by realising these spaces as boundaries of certain \emph{plumbings}, i.e.\ certain gluing operations of linear disc bundles, and showing that these spaces carry a Riemannian metric of positive Ricci curvature. This technique was later extended and generalized by Wraith \cite{Wr98,Wr07a}, resulting for example in connected sums of products of spheres whose dimensions may vary. Further constructions in this spirit are given in \cite{CW17}, \cite{Re22a,Re23,Re24} and \cite{Wr97,Wr11}. While these results greatly extend the number of known examples of manifolds with a Riemannian metric of positive Ricci curvature, they are limited to manifolds that can be expressed as boundaries of plumbings. In particular, this technique alone cannot provide an answer to Question \ref{Q:conn_sum}.
	
	A different approach to construct Riemannian metrics of positive Ricci curvature on connected sums was given by Perelman \cite{Pe97} for connected sums of copies of $\pm\C P^2$. This technique was later extended by Burdick \cite{Bu19} to manifolds admitting metrics of the following form:	
	\begin{definition}\label{D:core}
		A Riemannian metric $g$ on an $n$-manifold $M^n$ is a \emph{core metric} if it has positive Ricci curvature and if there exists an isometric embedding $D^n\hookrightarrow (M,g)$, where we consider $D^n$ as equipped with the induced metric of a hemisphere in the round sphere of radius $1$.
	\end{definition}
	
	We note that this definition slightly differs from that introduced in \cite{Bu19}, see Lemma \ref{L:core_equ} below. The significance of this notion with regard to Question \ref{Q:conn_sum} is shown by the following result:
	\begin{theorem}[{\cite{Pe97},\cite[Theorem B]{Bu19}}]\label{T:conn_sum}
		Let $M_1^n, \dots,M_\ell^n$ be $n$-manifolds that admit core metrics. Then the connected sum $M_1\#\dots\# M_\ell$ admits a Riemannian metric of positive Ricci curvature.
	\end{theorem}
	
	A consequence of this result in combination with the theorem of Bonnet--Myers, which also follows from \cite[Theorem 1]{La70}, is that a closed manifold that admits a core metric is simply-connected. Further topological obstructions to the existence of a core metric on a closed manifold that admits a Riemannian metric of positive Ricci curvature are not known.
	
	While Theorem \ref{T:conn_sum} offers a promising attempt towards answering Question \ref{Q:conn_sum}, constructing a core metric on a given Riemannian manifold of positive Ricci curvature is a non-trivial problem. The known examples of manifolds with core metrics are given as follows:
	\begin{enumerate}
		\item[\mylabel{EQ:core1}{(C1)}] The sphere $S^n$, $n\geq 2$, and the projective spaces $\C P^n$, $\Quat P^n$ and $\mathbb{O}P^2$ (\cite{Bu19}, \cite{Pe97}),
		\item[\mylabel{EQ:core2}{(C2)}] Total spaces of linear sphere bundles $E^{p+q}\to B^q$, where $B$ is closed and admits a core metric, $p\geq 2$ and $p+q\geq 6$ (\cite[Theorem B]{Bu20}, \cite[Theorem C]{Re23}),
		\item[\mylabel{EQ:core3}{(C3)}] Connected sums $M_1^n\# M_2^n$ where both $M_1$ and $M_2$ admit a core metric (\cite[Corollary 3.18]{Bu20a}), 
		\item[\mylabel{EQ:core4}{(C4)}] Boundaries of certain plumbings (\cite[Theorem C]{Bu19a}).
	\end{enumerate}
	
	Note that the case of Cartesian products $M_1\times M_2$ only appears in \ref{EQ:core2} in the special case where one of $M_1$ and $M_2$ is a sphere. More generally, one can ask whether core metrics exist on the total space of a fibre bundle whose base and fibre both admit a core metric, in analogy with a classical result for Riemannian metrics of positive Ricci curvature, see \cite{Po75}, \cite{Na79}, \cite[Theorem 2.7.3]{GW09} and Proposition \ref{P:Ric_bundles} below. The main result of this article is to answer this question affirmatively:
	
	\begin{theoremAlph}\label{T:core_bdl}
		Let $E\to B^q$ be a fibre bundle with fibre $F^p$ and structure group $G$ such that $E$ is closed and $q\geq 3$, $p\geq 2$. Suppose the following:
		\begin{enumerate}
			\item $B$ admits a core metric $\check{g}$,
			\item $F$ admits a Riemannian metric $\hat{g}$ of positive Ricci curvature that is invariant under the action of $G$,
			\item $F$ admits a core metric $\hat{g}'$ that lies in the same path component as $\hat{g}$ in the space of Ricci-positive metrics on $F$.
		\end{enumerate}
		Then $E$ admits a core metric.
	\end{theoremAlph}
	Note that Theorem \ref{T:core_bdl} in particular applies when $E=B^q\times F^p$ with $p\geq 2$, $q\geq 3$, and both $B$ and $F$ admit a core metric.
	
	We apply Theorem \ref{T:core_bdl} in two special cases: linear sphere bundles and projective bundles with fibre $\C P^n$, $\Quat P^n$ or $\mathbb{O} P^2$. By the latter we mean here a fibre bundle with structure group contained in $U(n+1)$, $Sp(n+1)$ or $F_4$, respectively. By \ref{EQ:core2}, for linear sphere bundles we are left with the case of a total space of dimension $4$ or $5$, while the case of a projective bundle is open. In fact, our techniques apply to total spaces of sphere and projective bundles without any restriction on the dimensions involved (provided both base and fibre admit a core metric), which in particular provides a new proof of \ref{EQ:core2}.	
	
	\begin{theoremAlph}\label{T:sph_proj_bdl}
		Let $E\xrightarrow{\pi}B^q$ be a linear sphere bundle with fibre $S^p$, $p\geq 2$, or a projective bundle with fibre $\C P^n$, $\Quat P^n$ or $\mathbb{O}P^2$. If $B$ is closed and admits a core metric, then $E$ admits a core metric.
	\end{theoremAlph}
	
	In particular, Theorem \ref{T:sph_proj_bdl} establishes the existence of a core metric on the manifold $S^2\times S^2$. This answers \cite[Question after Corollary A]{Bu20} affirmatively and shows that all closed, simply-connected 4-manifold that are known to admit a Riemannian metric of positive Ricci curvature, which consist of connected sums of copies of $S^2\times S^2$ and $\pm\C P^2$ (see \cite{SY93}), in fact admit a core metric. This also provides many new examples of manifolds with core metrics in higher dimensions, such as total spaces of linear $S^2$-bundles over $S^2\times S^2$, which all are new examples in dimension 6, see \cite[Remark 4.17]{Re22a} and \cite[Proposition 6.5.7]{Re22}.
	
	
	It remains open whether Theorem \ref{T:sph_proj_bdl} extends to the case $p=1$ as considered in \cite{BB78}, \cite{GPT98}.
	
	Our techniques also apply to spaces that are \emph{double disc bundles}, i.e.\ spaces obtained by gluing two linear disc bundles along their boundaries. We illustrate this by constructing a core metric on the \emph{Wu manifold}.
	
	\begin{theoremAlph}\label{T:Wu}
		The Wu manifold $W^5$ admits a core metric.
	\end{theoremAlph}
	
	Recall that the Wu manifold $W^5$ is a closed, simply-connected non-spin 5-manifold with $H_2(W)\cong\Z/2$. It can be described as the homogeneous space $SU(3)/SO(3)$, which, by classical results of Nash \cite{Na79}, admits a Riemannian metric of positive Ricci curvature. By \cite{DGK23} the Wu manifold is in fact the only closed, simply-connected 5-manifold that is a double disc bundle besides the sphere $S^5$, the product $S^2\times S^3$ and the total space of the unique non-trivial linear $S^3$-bundle over $S^2$. Moreover, it is one of the \textquotedblleft elementary\textquotedblright\ manifolds in Barden's classification of closed, simply-connected 5-manifolds \cite{Ba65}. To the best of our knowledge, the only known examples of such manifolds with a metric of positive Ricci curvature are as follows:
	\begin{enumerate}
		\item All closed, simply-connected 5-manifolds with torsion-free homology (\cite{SY91}, see also \cite{CG20}),
		\item Certain closed, simply-connected spin $5$-manifolds with second Betti number at most 8 using Sasakian geometry (\cite{Ko09}, \cite{BG02,BG06a} and \cite[Corollary 10.2.20, Theorem 10.2.25 and Table B.4.2]{BG08}),
		\item The Wu manifold $W$.
	\end{enumerate}
	In particular, the Wu manifold is the only known non-spin example with torsion in its homology. By Theorems \ref{T:sph_proj_bdl} and \ref{T:Wu} we can now construct examples of closed, simply-connected non-spin 5-manifolds with any given second Betti number and arbitrarily large torsion group by taking connected sums of the form
	\[ \#_{\ell_1}W\#_{\ell_2}(S^2\times S^3). \]
	
	We note that the proof of Theorem \ref{T:Wu} can be generalized to all dimension $4m+1$, see Remark \ref{R:Wm} below.	
	
	Finally, we show that we can use Theorem \ref{T:sph_proj_bdl} to construct manifolds of positive Ricci curvature that realize many different bordism classes.
	
	\begin{corollaryAlph}\label{C:bordism}
		Any class in the torsion-free oriented bordism ring $\bigslant{\Omega_*^{SO}}{Tors}$ is represented by a connected manifold that admits a Riemannian metric of positive Ricci curvature.
	\end{corollaryAlph}
	Corollary \ref{C:bordism} will directly follow from Theorem \ref{T:sph_proj_bdl} after showing that a generating set can be realized by manifolds admitting core metrics. We note that one motivation for Corollary \ref{C:bordism} is the fact that the aforementioned plumbing construction, which is one of the main tools to construct Riemannian metrics of positive Ricci curvature, typically results in manifolds that are boundaries, i.e.\ that are trivial in the corresponding bordism group.
	
	It is open whether a similar result also holds for the spin bordism ring $\Omega_*^{Spin}$ when restricting to the kernel of the $\hat{A}$-genus. A partial result in this direction is given in Proposition \ref{P:spin_bord} below. Further, it is shown in \cite[Proposition 3.5]{De09} that in every dimension $n\leq 100$ a generating set of $\ker(\hat{A}\colon\Omega_n^{Spin}\otimes\Q\to\Q)$ is represented by manifolds admitting a Riemannian metric of positive Ricci curvature. One can also ask whether it is possible to extend Corollary \ref{C:bordism} to stricter curvature conditions such as a lower bound on the sectional curvature (\cite{DT07}, \cite{HW18}).	
	
	To illustrate the challenges in proving Theorem \ref{T:core_bdl}, consider the case of a trivial bundle $B^q\times F^p$. If $D^q\subseteq B$ and $D^p\subseteq F$ denote the embedded hemispheres, we can decompose the space $(B\times F)\setminus {D^{p+q}}^\circ$ as
	\[ (B\times F)\setminus {D^{p+q}}^\circ\cong (B\setminus {D^q}^\circ)\times F\cup_{S^{q-1}\times (F\setminus{D^p}^\circ)} D^q\times (F\setminus{D^p}^\circ). \]
	Since the boundaries of $(B\setminus{D^q}^\circ)$ and $(F\setminus{D^p}^\circ)$ are totally geodesic, the same holds for the boundaries of $(B\setminus {D^q}^\circ)\times F$ and $D^q\times (F\setminus{D^p}^\circ)$. Hence, if we could smooth out the corner $S^{q-1}\times S^{p-1}$ obtained by the gluing while preserving a non-negative second fundamental form, we could attach a round hemisphere to $(B\times F)\setminus {D^{p+q}}^\circ$ after a deformation of the metric on the boundary. However, since the dihedral angles at the corner are given by $\frac{3}{2}\pi>\pi$, smoothing in a small neighbourhood results in highly negative principal curvatures.
	
	We will therefore construct a specific metric on the cylinder $I\times S^{q-1}\times F$ which we glue in-between, that reduces the angle to less than $\pi$. For that we will cut out a piece of a metric that is close to the product of a round metric on $S^{q-1}$ and a cone metric over $F$ along a hypersurface that resembles a geodesic in the cone over $S^1$ starting tangent to $S^1$ and approaching a line orthogonal to $S^1$ as $t\to\infty$ (see also Figure \ref{F:handle1} below). For this to be possible we need a sufficient slope on the cone, i.e.\ a certain lower bound on the principal curvatures for directions tangent to $F$. For this we will use the doubly warped product constructions of \cite{Re23} to \textquotedblleft transfer\textquotedblright\ positive principal curvatures from the $S^{q-1}$-factor to $F$.
	
	This article is laid out as follows. In Section \ref{S:PREL} we recall basic facts on the second fundamental form and the Ricci curvature of certain metrics on a cylinder (Subsection \ref{SS:cylinder}) and spaces obtained by cutting along the graph of a function (Subsection \ref{SS:cut_graph}), discuss Riemannian metrics of positive Ricci curvature on fibre bundles (Subsection \ref{SS:fibre_bdls}) and collect various results that allow to deform and glue Riemannian metrics of positive Ricci curvature (Subsection \ref{SS:glue_deform}). In Section \ref{S:build_block} we construct the \textquotedblleft building block\textquotedblright\ which is the key object in the proof of Theorem \ref{T:core_bdl}. In Section \ref{S:glued_spaces} we then prove a general result on core metrics on certain glued spaces and use it to prove Theorem \ref{T:core_bdl}. Finally, in Section \ref{S:app} we consider applications and prove Theorem \ref{T:sph_proj_bdl} in Subsections \ref{SS:sph_bdls} and \ref{SS:Proj_bdl}, Theorem \ref{T:Wu} in Subsection \ref{SS:Wu} and Corollary \ref{C:bordism} in Subsection \ref{SS:bordism}.
	
	\begin{ack}
		The author would like to thank Anand Dessai, Fernando Galaz-García, Wilderich Tuschmann and David Wraith for helpful conversations and comments on an earlier version of this article.
	\end{ack}
	
	\section{Preliminaries}\label{S:PREL}
	
	Let $(M^n,g)$ be a Riemannian manifold and $N\subseteq M$ an embedded hypersurface with trivial normal bundle and let $\nu$ be a unit normal field on $N$. The \emph{second fundamental form} of $N$ at $x\in N$ is defined by
	\[ \II(u,v)=g(\nabla_u \nu,v) \]
	for $u,v\in T_x N$, where $\nabla$ denotes the Levi-Civita connection of $M$. If $N$ is a boundary component of $M$, we use the convention that $\II$ is calculated with respect to the outward pointing unit normal field. The \emph{principal curvatures} of the boundary are then the eigenvalues of $\II$. We say that the boundary of $(M,g)$ is \emph{convex} if at every point all principal curvatures are positive, i.e.\ if the second fundamental form is positive definite. For example, a geodesic ball in the round sphere of radius 1 has convex boundary if and only if it is a proper subset of a hemisphere. We refer to \cite[Section 3.2.1]{Pe16} for further information on the second fundamental form. An important aspect will be the behaviour of $\II$ under scaling of the metric $g$. Since the Levi--Civita connection is unaffected by scaling, we obtain that the second fundamental form $\II^{R^2g}$ of $R^2g$ for $R>0$ is given by
	\[\II^{R^2g}=R\,\II^g \]
	and the principal curvatures are scaled by $\tfrac{1}{R}$.

	\subsection{Metrics on a cylinder}\label{SS:cylinder}
	
	In this subsection we consider metrics on a cylinder $I\times M$ and establish formulas for the curvatures of such a metric.
	
	\begin{lemma}
		\label{L:curv_form}
		Let $M$ be a manifold, $I$ an interval and let $g=dt^2+g_t$ be a Riemannian metric on $I\times M$, where $g_t$ is a smooth family of metrics on $M$. Let $g_t'$ and $g_t''$ denote the first and second derivative of $g_t$ in $t$-direction, respectively. Then the second fundamental form of a slice $\{t\}\times M$ with respect to the normal vector $\partial_t$ is given by
		\[\II=\frac{1}{2}g_t' \]
		and the Ricci curvature of the metric $g$ are given as follows:
		\begin{align*}
			&\Ric(\partial_t,\partial_t)=-\frac{1}{2}\mathrm{tr}_{g_t}g_t^{\prime\prime}+\frac{1}{4}\lVert g_t^\prime\rVert_{g_t}^2,\\
			&\Ric(v,\partial_t)\;=-\frac{1}{2}v(\mathrm{tr}_{g_t}g_t')+\frac{1}{2}\sum_{i}(\nabla^{g_t}_{e_i}g_t')(v,e_i),\\
			&\Ric(u,v)\;\;=\Ric^{g_t}(u,v)-\frac{1}{2}g_t^{\prime\prime}(u,v)+\frac{1}{2}\sum_i g_t^\prime(u,e_i)g_t^\prime(v,e_i)-\frac{1}{4}g_t^\prime(u,v)\mathrm{tr}_{g_t}g_t^\prime.
		\end{align*}
		Here $u,v\in T_xM$ and $\{e_i\}$ is an orthonormal basis of $T_xM$ with respect to $g_t$.
	\end{lemma}
	\begin{proof}
		We extend $u$ and $v$ along normal coordinates in $M$ around $x$ and constant in $t$-direction. Then $[u,\partial_t]=[v,\partial_t]=0$ by the product structure and hence the second fundamental form of the hypersurface $\{t\}\times M$ is given by
		\begin{align*}
			\II(u,v)&=\tfrac{1}{2}(\II(u,v)+\II(v,u))=\tfrac{1}{2}(g_t(\nabla_u \partial_t,v)+g_t(u,\nabla_v \partial_t))=\tfrac{1}{2}(g_t(\nabla_{\partial_t} u,v)+g_t(u,\nabla_{\partial_t} v))\\
			&=\tfrac{1}{2}\left(\tfrac{\partial}{\partial t}g_t(u,v)-g_t(u,\nabla_{\partial_t} v)+g_t(u,\nabla_{\partial_t} v)\right)=\tfrac{1}{2}g_t^\prime(u,v).
		\end{align*}
		From the tangential curvature equation we now obtain for $w\in T_xM$
		\begin{align}
			\label{EQ:TANG_CURV}
			g(R(u,v)v,w)=g_t(R^{g_t}(u,v)v,w)-\frac{1}{4}(g_t^\prime(u,w)g_t^\prime(v,v)-g_t^\prime(u,v)g_t^\prime(w,v)).
		\end{align}
		Next we apply the normal curvature equation and use that $u$ and $v$ are vector fields with respect to normal coordinates around $x$, so that any covariant derivatives involving $u$ and $v$ vanish at $x$. Hence, we obtain:
		\begin{align}
			\notag g(R(u,v)v,\partial_t)&=-(\nabla_u\II)(v,v)+(\nabla_v\II)(u,v)\\
			\notag&=-u(\II(v,v))+\frac{1}{2}(\nabla_vg_t')(u,v)\\
			\label{EQ:NORMAL_CURV}&=-\frac{1}{2}u(g_t'(v,v))+\frac{1}{2}(\nabla_vg_t')(u,v).
		\end{align}
		Denote by $S=\nabla_{\cdot}\partial_t$ the shape operator of the hypersurface $\{t\}\times M$, i.e.\ $\II=g_t(S\cdot,\cdot)$. Then, by the radial curvature equation, we have
		\begin{align}
			\label{EQ:RADIAL_CURV1}
			\notag g(R(u,\partial_t)\partial_t,v)&=-g_t(S^2(u),v)-g_t((\nabla_{\partial_t} S)(u),v)\\
			\notag &=-g_t(S^2(u),v)-g_t(\nabla_{\partial_t}(S(u)),v)+g_t(S(\nabla_{\partial_t} u),v)\\
			\notag &=-g_t(S^2(u),v)-\frac{\partial}{\partial t}g_t(S(u),v)+g_t(S(u),\nabla_{\partial_t} u)+g_t(S(\nabla_u \partial_t),v)\\
			\notag &=-\frac{\partial}{\partial t}\II(u,v)+g_t(S(u),S(v))\\
			\notag &=-\frac{1}{2}g_t^{\prime\prime}(u,v)+g_t\left(\sum_i g_t(S(u),e_i)e_i,\sum_j g_t(S(v),e_j)e_j\right)\\
			&=-\frac{1}{2}g_t^{\prime\prime}(u,v)+\frac{1}{4}\sum_i g_t^\prime(u,e_i)g_t^\prime(v,e_i).
		\end{align}
		Finally, we have
		\begin{align}
			\label{EQ:RADIAL_CURV2}
			g(R(u,\partial_t)\partial_t,\partial_t)=0
		\end{align}
		by the skew symmetry property of the Riemann curvature tensor.
		
		The formulas for the Ricci curvatures now directly follow from \eqref{EQ:TANG_CURV}--\eqref{EQ:RADIAL_CURV2}.
	\end{proof}
	

	We obtain the following consequence (see also \cite[Section 4.2.4]{Pe16}).

	\begin{lemma}\label{L:doubly_warped_curv}
		Let $g=dt^2+f(t)^2ds_p^2+h(t)^2ds_q^2$ be a doubly warped product metric on $[0,t_0]\times S^p\times S^q$. Let $u,u_1,u_2$ and $v,v_1,v_2$ be unit tangent vectors of $(S^p,ds_p^2)$ and $(S^q,ds_q^2)$, respectively. Then the sectional curvatures of $g$ are given as follows:
		\begin{align*}
			\sec(\partial_t\wedge u)&=-\frac{f''}{f},\\
			\sec(\partial_t\wedge v)&=-\frac{h''}{h},\\
			\sec(u\wedge v)&=-\frac{f'h'}{fh},\\
			\sec(u_1\wedge u_2)&=\frac{1-{f'}^2}{f^2},\\
			\sec(v_1\wedge v_2)&=\frac{1-{h'}^2}{h^2}
		\end{align*}
		and all other sectional curvatures are convex combinations of these expressions. Further, the Ricci curvatures of $g$ are given as follows:
		\begin{align*}
			\Ric(\partial_t,\partial_t)&=-p\frac{f''}{f}-q\frac{h''}{h},\\
			\Ric(\tfrac{u}{f},\tfrac{u}{f})&=-\frac{f''}{f}+(p-1)\frac{1-{f'}^2}{f^2}-q\frac{f'h'}{fh},\\
			\Ric(\tfrac{v}{h},\tfrac{v}{h})&=-\frac{h''}{h}+(q-1)\frac{1-{h'}^2}{h^2}-p\frac{f'h'}{fh},\\
			\Ric(\partial_t,u)&=\Ric(\partial_t,v)=\Ric(u,v)=0.
		\end{align*}
	\end{lemma}

	Given a doubly warped product metric $g=dt^2+f(t)^2ds_p^2+h(t)^2ds_q^2$, we can collapse the sphere $S^p$ at $t=0$ and the sphere $S^q$ at $t=t_0$ to obtain the sphere $S^{p+q+1}$. The metric $g$ then defines a smooth metric on this collapsed space if and only if
	\begin{enumerate}
		\item[\mylabel{EQ:dw_boundary_0}{(DW1)}] $f$ is an odd function at $t=0$ with $f'(0)=1$ and $h$ is an even function at $t=0$,
		\item[\mylabel{EQ:dw_boundary_t0}{(DW2)}] $f$ is an even function at $t=t_0$ and $h$ is an odd function at $t=t_0$ with $h'(t_0)=-1$,
	\end{enumerate}
	see e.g.\ \cite[Propositions 1.4.7 and 1.4.8]{Pe16}. For example, if we set $f(t)=\tfrac{2t_0}{\pi}\sin(\tfrac{\pi}{2t_0}t)$ and $h(t)=\tfrac{2t_0}{\pi}\cos(\tfrac{\pi}{2t_0}t)$, the metric $g$ is a scalar multiple of the round metric on $S^{p+q+1}$.
	
	\begin{lemma}\label{L:sphere_dbl_warped}
		A doubly warped product metric $g=dt^2+f(t)^2ds_p^2+h(t)^2ds_q^2$ on $S^{p+q+1}$ has positive sectional curvature if and only if $f''$ and $h''$ are strictly negative on $(0,t_0]$ and $[0,t_0)$, respectively, and $f'''(0)$, $h'''(t_0)<0$. In particular, the space of doubly warped product metrics of positive sectional curvature on $S^{p+q+1}$ is path-connected.
	\end{lemma}
	\begin{proof}
		It follows from Lemma \ref{L:doubly_warped_curv} that a necessary condition of $g$ to have positive sectional curvature is given by $f''(t),h''(t)<0$ whenever $f(t), h(t)\neq 0$, and, by l'Hôpital's rule, $f'''(0),h'''(t_0)<0$. Conversely, if this is satisfied, it follows from the boundary conditions (1) and (2) that $f'\in[0,1]$ and $h'\in[-1,0]$, where the boundary points of the intervals are only attained at $t=0,t_0$. It follows from Lemma~\ref{L:doubly_warped_curv} that $g$ has positive sectional curvature on $(0,t_0)$ and, by applying l'Hôpital's rule, at $t=0,t_0$ as well.
		
		Finally, that the space of doubly warped product metrics on $S^{p+q+1}$ of positive sectional curvature is path-connected now follows from the fact that the properties $f'',h''<0$ on $(0,t_0]$ resp.\ $[0,t_0)$ and $f'''(0),h'''(t_0)<0$ as well as the boundary conditions (1) and (2) are preserved under convex combinations.
	\end{proof}

	\subsection{Metric induced by the graph of a function}\label{SS:cut_graph}
	
	Let $(M,g)$ be a Riemannian manifold, let $f\colon\R\to(0,\infty)$ be a smooth function and let $\rho\colon M\to\R$ be smooth. We consider the embedding
	\begin{align*}
		M&\hookrightarrow \R\times M\\
		x&\mapsto (\rho(x),x)
	\end{align*}
	and denote by $M^\prime$ the image of this embedding. Then $M'$ is a hypersurface in $\R\times M$. We equip $\R\times M$ with the metric $\overline{g}=dt^2+f(t)^2 g$ and we will be interested in the induced metric on $M^\prime$, as well as its second fundamental form.
	
	For that, first note that the differential of $\rho\times \mathrm{id}_M$ is given by $d\rho\times \mathrm{id}$, so
	\[TM^\prime=\{d\rho(X)\partial_t+X\mid X\in TM \}. \]
	Hence, the upward pointing normal vector field on $M^\prime$ is given by
	\[\nu=\frac{1}{\sqrt{1+\frac{\lVert\nabla\rho\rVert^2}{f^2}}}\left(\partial_t-\frac{1}{f^2}\nabla\rho\right). \]
	Here $\nabla\rho$ is the gradient with respect to $g$ and the norm is taken with respect to the metric $g$.
	
	\begin{lemma}\label{L:II_graph}
		For $x\in M$ the second fundamental form of $M'$ at $(\rho(x),x)$ with respect to $\nu$ is given as follows:
		\[ \II(X(\rho)\partial_t+X,Y(\rho)\partial_t+Y)=\frac{1}{\sqrt{1+\frac{\rVert\nabla\rho\rVert^2}{f^2}}}\left( f'fg(X,Y)+2X(\rho)Y(\rho)\frac{f'}{f}-XY(\rho)+\nabla_X Y(\rho) \right). \]
	\end{lemma}
	\begin{proof}
		Let $X$ and $Y$ be local vector fields in $M$, which we extend constantly in $t$-direction to local vector fields of $\R\times M$. From the Koszul formula we now obtain the following:
		\begin{align*}
			\overline{\nabla}_{\partial_t}\partial_t&=0\\
			\overline{\nabla}_{X}Y&=-f^\prime fg(X,Y)\partial_t+\nabla_{X}Y\\
			\overline{\nabla}_{\partial_t}X&=\overline{\nabla}_{X}\partial_t=\frac{f^\prime}{f}X.
		\end{align*}
		Here $\overline{\nabla}$ and $\nabla$ denote the Levi--Civita connections of $\overline{g}$ and $g$, respectively.
		
		A calculation then shows the following:
		\begin{align*}
			\overline{\nabla}_{X}\nu&=\frac{1}{\sqrt{1+\frac{\lVert\nabla\rho\rVert^2}{f^2}}}\left(-\frac{1}{f^2+\lVert\nabla\rho\rVert^2}g\left(\nabla_{X}\nabla\rho,\nabla\rho\right)\left(\partial_t-\frac{1}{f^2}\nabla\rho\right)+\frac{f^\prime}{f}X+\frac{f'}{f}X(\rho)\partial_t-\frac{1}{f^2}\nabla_X\nabla\rho \right),\\
			\overline{\nabla}_{\partial_t}\nu&=\frac{1}{\sqrt{1+\frac{\lVert\nabla\rho\rVert^2}{f^2}}}\left(\frac{\lVert\nabla\rho\rVert^2}{f^2+\lVert\nabla\rho\rVert^2}\frac{f^\prime}{f}\left(\partial_t-\frac{1}{f^2}\nabla\rho\right)+\frac{f^\prime}{f^3}\nabla\rho\right).
		\end{align*}
		Hence, we obtain
		\begin{align*}
			\II(X(\rho)\partial_t+X,Y(\rho)\partial_t+Y)=&\,\overline{g}\left(\overline{\nabla}_{X(\rho)\partial_t+X}\nu,Y(\rho)\partial_t+Y\right)\\
			=&\,\overline{g}\left(\overline{\nabla}_X \nu,Y(\rho)\partial_t+Y\right)+X(\rho)\overline{g}\left(\overline{\nabla}_{\partial_t}\nu,Y(\rho)\partial_t+Y\right) \\
			=&\,\frac{1}{\sqrt{1+\frac{\lVert\nabla\rho\rVert^2}{f^2}}}\left( f'fg(X,Y)+X(\rho)Y(\rho)\frac{f'}{f}-g(\nabla_X\nabla\rho,Y)+X(\rho)Y(\rho)\frac{f'}{f} \right)\\
			=&\,\frac{1}{\sqrt{1+\frac{\lVert\nabla\rho\rVert^2}{f^2}}}\left( f'fg(X,Y)+2X(\rho)Y(\rho)\frac{f'}{f}-XY(\rho)+\nabla_X Y(\rho) \right).
		\end{align*}
	\end{proof}
	
	We now restrict to the case where $M=S^n$ and $g$ is a warped product metric of the form
	\[ g=ds^2+R(s)^2ds_{n-1}^2, \]
	where $s\in[0,s_0]$ and $R\colon[0,s_0]\to[0,\infty)$ is an odd function at $s=0$ and $s=s_0$ with $R'(0)=1$, $R'(s_0)=-1$ and $R|_{(0,s_0)}>0$. We also assume that $\rho$ now only depends on $s$, i.e.\ there exists $\alpha\colon[0,s_0]\to\R$ with $\rho(s,x)=\alpha(s)$ for all $(s,x)\in[0,s_0]\times S^{n-1}$. Then $\nabla\rho=\alpha'\partial_s$, so we have
	\[ TM'=\langle \partial_s+\alpha'\partial_t\rangle\oplus TS^{n-1}. \]
	Then we obtain from Lemma \ref{L:II_graph} the following expressions for the second fundamental form.
	\begin{lemma}\label{L:II_graph_warped}
		In the warped product case $g=ds^2+R(s)^2ds_{n-1}^2$ we have the following, where $X$ and $Y$ are tangent vectors tangent to $S^{n-1}$:
		\begin{align*}
			\II(\partial_s+\alpha'\partial_t,\partial_s+\alpha'\partial_t)&=\frac{1}{\sqrt{1+\frac{{\alpha'}^2}{f^2}}}\left( f'f+2\frac{f'}{f}{\alpha'}^2-\alpha'' \right),\\
			\II(\partial_s+\alpha'\partial_t,X)&=0,\\
			\II(X,Y)&=\frac{g(X,Y)}{\sqrt{1+\frac{{\alpha'}^2}{f^2}}}\left( f'f-\frac{R'}{R}\alpha' \right).
		\end{align*}
	\end{lemma}
	
	For the construction in Section \ref{S:build_block} the following choice of $f$ will be useful.
	\begin{lemma}\label{L:fH}
		Let $0<\lambda_1<\lambda_2<1$ and $\delta>0$ sufficiently small. Then there exists a smooth function $f\colon[-\delta,\infty)\to(0,\infty)$ with the following properties:
		\begin{enumerate}
			\item $f(0)=1$, $f'(0)=\lambda_2$,
			\item $f''<0$,
			\item $f'>\lambda_1$ and $\frac{f'(t)}{f(t)}>\frac{\lambda_1}{1+\lambda_1 t}$.
		\end{enumerate}
	\end{lemma}
	\begin{proof}
		We set $f(t)=(\lambda_2-\lambda')H(t)+\lambda' t+1$ for some function $H\colon[-\delta,\infty)\to(0,\infty)$ and $\lambda'\in(\lambda_1,\lambda_2)$. A calculation shows that $f$ satisfies the required properties provided $H$ satisfies
		\begin{enumerate}
			\item $H(0)=0$, $H'(0)=1$,
			\item $H'>0$, $H''<0$,
			\item $H'(t)(1+\lambda_1 t)-\lambda_1 H(t)>\frac{\lambda_1-\lambda'}{\lambda_2-\lambda'}$.
		\end{enumerate}
		This is for example satisfied for $H(t)=1-e^{-t}$ and $\lambda'>\lambda_1\frac{1+\lambda_2}{1+\lambda_1}$.
	\end{proof}
	
	\subsection{Metrics of positive Ricci curvature on fibre bundles}\label{SS:fibre_bdls}
	
	In this subsection we review general facts on submersion metrics on fibre bundles in the context of positive Ricci curvature and extend the main construction of \cite{Re23} for doubly warped product metrics to the case of a (possibly non-trivial) fibre bundle. General references are \cite{Be87}, \cite{GW09}, \cite{Na79} and \cite{Po75}.
	
	Let $E\xrightarrow{\pi} B^q$ be a fibre bundle with fibre $F^p$ and structure group $G$. The \emph{vertical subbundle} of $TE$, denoted $\mathcal{V}$, is defined by $\mathcal{V}=\ker(d\pi)$. A subbundle $\mathcal{H}\subseteq TE$ of rank $q$ which is complementary to $\mathcal{V}$ at every point is a \emph{horizontal distribution}. Any horizontal distribution $\mathcal{H}$ is canonically isomorphic to $\pi^*TB$, thus defining a splitting
	\[ TE\cong \pi^*TB\oplus \mathcal{V}. \]
	For example, any Riemannian metric on $E$ induces a horizontal distribution via $\mathcal{H}=\mathcal{V}^\perp$.	
	
	Let $P\to B$ be the principal $G$-bundle associated to $\pi$, i.e.\ there is a bundle isomorphism $E\cong P\times_G F$. Then any principal connection $\theta\in\Omega^1(P)$ induces a horizontal distribution on $E$ obtained as the image of $\ker(\theta)\times TF\subseteq TP\times TF$ in $T(P\times_G F)$.
		
	Now given a Riemannian metric $\check{g}$ on $B$, a $G$-invariant Riemannian metric $\hat{g}$ on $F$ and a principal connection $\theta$ on $P$, there exists precisely one Riemannian metric $g$ on $E$ such that $(E,g)\to(B,\check{g})$ is a Riemannian submersion with totally geodesic fibres isometric to $(F,\hat{g})$ and horizontal distribution induced by $\theta$. This result is known as the \emph{Vilms construction}, see e.g.\ \cite[Theorem 9.59]{Be87}, \cite{Vi70}. The metric $g$ can explicitly be written as
	\begin{equation}
		g=\pi^*\check{g}+\mathrm{pr}_{\mathcal{V}}^*\hat{g},\label{EQ:submersion_metric}
	\end{equation}
	where $\mathrm{pr}_{\mathcal{V}}\colon TE\to\mathcal{V}$ denotes the projection onto $\mathcal{V}$ along the splitting $TE=\mathcal{H}\oplus\mathcal{V}$ and we identify $F$ with its image under its inclusion along a local trivialization.
	
	We will consider shrinking the metric in fibre direction, i.e.\ replacing $\hat{g}$ by $r^2\hat{g}$ for $r>0$. We denote the metric obtained in this way by $g_r$, or $g_r^\theta$ to specify the principal connection. Then the Ricci curvatures of $g_r$ are given as follows, see e.g.\ \cite[Proposition 9.70]{Be87}:
	\begin{align}
		\notag \Ric(U,V)&=\Ric_{\hat{g}}(U,V)+r^4(AU,AV),\\
		\Ric(U,X)&=r^2((\check{\delta}A)(X),U)\label{EQ:Ric_bundles}\\
		\notag \Ric(X,Y)&=\Ric_{\check{g}}(X,Y)-2r^2(A_X,A_Y).
	\end{align}
	Here $U,V$ and $X,Y$ are vertical and horizontal vectors, respectively, and $A$ denotes the $A$-tensor of $g_1$ (see e.g.\ \cite[Section 9.C]{Be87}). We obtain the following consequence:
	\begin{proposition}[{\cite[Proposition 9.70]{Be87}, \cite[Theorem 2.7.3]{GW09}}]\label{P:Ric_bundles}
		If $E$ is compact and $\check{g}$ and $\hat{g}$ have positive Ricci curvature, then there exists $r_0>0$ such that $g_r$ has positive Ricci curvature for all $r<r_0$.
	\end{proposition}
	
	We will now assume that $E$ is compact and denote by $g_r$ the metric obtained as in Proposition \ref{P:Ric_bundles} for some fixed metrics $\check{g}$ and $\hat{g}$ and principal connection $\theta$. The following result is a generalisation of the main construction in \cite[Section 3]{Re23} and will allow to \textquotedblleft transfer\textquotedblright\ a positive second fundamental form from the fibre to the base on the cylinder $I\times E$.
	
	\begin{proposition}\label{P:bdl_warping}
		Suppose that $\Ric_{\check{g}}\geq (q-1)\check{g}$ with $q\geq 2$ and $\Ric_{\hat{g}}>0$. Then, for any $r_0,r_1,\nu>0$ and $\lambda\in(0,1)$ there exist constants $\kappa_0=\kappa_0(p,q,A,\hat{g},\lambda)>0$ and $\kappa_1=\kappa_1(p,q,r_0,\nu,A,\hat{g},\lambda)>0$ such that if $r_0<\kappa_0$ and $r_1<\kappa_1$, then there is a metric of positive Ricci curvature on $[0,1]\times E$ such that the following holds:
		\begin{enumerate}
			\item The metric on $\{0\}\times E$ is isometric to $g_{r_0}$ and the second fundamental form satisfies
			\begin{align*}
				\II_{\{0\}\times E}(V,V)&\geq -\nu g_{r_0}(V,V),\\
				\II_{\{0\}\times E}(X,V)&=0,\\
				\II_{\{0\}\times E}(X,X)&=0,
			\end{align*}
			\item The metric on $\{1\}\times E$ is isometric to $R^2 g_{r_1}$ for some $R>0$ and the second fundamental form satisfies
			\begin{align*}
				\II_{\{1\}\times E}(V,V)&\geq 0,\\
				\II_{\{1\}\times E}(X,V)&=0,\\
				\II_{\{1\}\times E}(X,X)&\geq \lambda R g_{r_1}(X,X).
			\end{align*}
		\end{enumerate}
		Here $X$ resp.\ $V$ denote horizontal resp.\ vertical tangent vectors of $E$.
	\end{proposition}
	\begin{proof}
		For $t_0>0$ and smooth functions $f,h\colon[0,t_0]\to(0,\infty)$ we define the metric $g_{f,h}$ on $[0,t_0]\times E$ by
		\[ g_{f,h}=dt^2+f(t)^2\pi^*\check{g}+h(t)^2\mathrm{pr}_{\mathcal{V}}^*\hat{g}. \]
		By \cite[Proposition 4.2]{Wr07}, the Ricci curvatures of this metric are given as follows:
		\begin{align*}
			\Ric(\partial_t,\partial_t)&= -q\frac{f''}{f}-p\frac{h''}{h},\\
			\Ric(\tfrac{X}{f},\tfrac{X}{f})&=-\frac{f''}{f}+\frac{\Ric_{\check{g}}(X,X)-(q-1){f'}^2}{f^2}-p\frac{f'h'}{fh}-2\frac{h^2}{f^4}(A_X,A_X),\\
			\Ric(\tfrac{V}{h},\tfrac{V}{h})&=-\frac{h''}{h}+\frac{\Ric_{\hat{g}}(V,V)-(p-1){h'}^2 }{h^2}-q\frac{f'h'}{fh}+\frac{h^2}{f^4}(AV,AV),\\
			\Ric(\tfrac{X}{f},\tfrac{V}{h})&=-\frac{h}{f^3}g((\check{\delta}A)X,V),\\
			\Ric(\partial_t,\tfrac{X}{f})&=\Ric(\partial_t,\tfrac{V}{h})=0.
		\end{align*}
		Here $X$ and $V$ are horizontal and vertical unit vectors with respect to $\check{g}$ and $\hat{g}$, respectively.
		
		Further, the second fundamental form of a slice $\{t\}\times E$ with respect to the normal vector $\partial_t$ is given by
		\[ \II_{\{t\}\times E}=f'f\pi^*\check{g}+h'h\,\mathrm{pr}_{\mathcal{V}}^*\hat{g}, \]
		see e.g.\ \cite[Lemma 2.7]{Bu19}.
		
		We will now construct functions $f$ and $h$ so that the Ricci curvatures are positive and such that the function $g_{f,h}$ satisfies the required boundary conditions. For that we will use the warping functions constructed in \cite[Section 3.3]{Re23}. For a parameter $C>0$ these are functions $h_0,f_C\colon[0,\infty)\to(0,\infty)$ with the following properties:
		\begin{enumerate}
			\item $f_C(0)=1$, $f_C'(0)=0$,
			\item $h_0'=e^{-\frac{1}{2}h_0^2}>0$ and $h_0''=-h_0e^{-h_0^2}<0$,
			\item $f_C'>0$ on $(0,\infty)$ with $\lim_{t\to\infty}f'_C(t)=\infty$ and $f_C''=Ce^{-h_0^2}f_C>0$,
			\item $\lim_{t\to\infty}f_C(t)h_0'(t)=0$,
			\item $\frac{f_C'}{f_Ch_0h_0'}\in[0,1]$.
		\end{enumerate}
		For $a>0$ we set $h_a=ah_0$ and $f_{a,C}=\frac{a}{c}f_C$, where $c=\frac{r_0}{h_0(0)}$. We will now show that for $a$ and $C$ sufficiently small there exists $t_0>0$ such that $g_{f,h}$ for $h=h_a$ and $f=f_{a,C}$ satisfies all required properties on $[0,t_0]\times E$ after rescaling.
		
		First note that, by (2) and (3), we obtain that $\Ric(\partial_t,\partial_t)$ is positive for $C$ sufficiently small. Next, we obtain the following expression for the Ricci curvatures in vertical direction:
		\begin{align*}
			\Ric(V,V)&=a^2h_0^2\left(e^{-h_0^2}+\frac{\Ric_{\hat{g}}(V,V)-a^2(p-1){h_0'}^2}{a^2h_0^2}-q\frac{f'_Ch_0'}{f_Ch_0}+\frac{c^4}{a^2}\frac{h_0^2}{f_C^4}(AV,AV)\right)\\
			&= a^2h_0^2e^{-h_0^2}+\Ric_{\hat{g}}(V,V)-(p-1)a^2h_0'^2-qa^2\frac{f_C'}{f_C h_0 h_0'}h_0^2e^{-h_0^2}+ c^4\frac{h_0^4}{f_C^4}(AV,AV).
		\end{align*}
		Since the last summand is non-negative, $\Ric_{\hat{g}}(V,V)$ is positive, and all other terms are uniformly bounded, it follows that $\Ric(V,V)$ is uniformly positive as $a\to 0$.
		
		Further, for the Ricci curvatures in horizontal direction we have the following:
		\begin{align*}
			\Ric(X,X)&=\frac{a^2}{c^2}f_C^2\left(-Ce^{-h_0^2}+\frac{\Ric_{\check{g}}(X,X)-(q-1)\frac{a^2}{c^2}{f_C'}^2}{\frac{a^2}{c^2}f_C^2}-(p-1)\frac{f_C'h_0'}{f_Ch_0}-2\frac{c^4}{a^2}\frac{h_0^2}{f_C^4}(A_X,A_X)\right)\\
			&= -\frac{a^2}{c^2}Ce^{-h_0^2}f_C^2+\left( \Ric_{\check{g}}(X,X)-(q-1)\frac{a^2}{c^2}{f_C'}^2 \right)-(p-1)\frac{a^2}{c^2}\frac{f_C'}{f_Ch_0h_0'}f_C^2{h_0'}^2-2c^2\frac{h_0^2}{f_C^2}(A_X,A_X).
		\end{align*}
		For $f_{a,C}'\in[0,\lambda]$, we have
		\[ \Ric_{\check{g}}(X,X)-(q-1)\frac{a^2}{c^2}{f_C'}^2\geq (q-1)(1-\lambda)>0.  \]
		All other terms are uniformly bounded, so that $\Ric(X,X)$ is uniformly positive if we choose $c$ (and therefore $r_0$) and subsequently $a$ sufficiently small.
		
		Finally, for the mixed Ricci curvatures we obtain
		\[ \Ric(X,V)=-c^2\frac{h_0^2}{f_C^2}g((\check{\delta}A)X,V). \]
		Since $\frac{h_0}{f_C}$ is bounded, we have $\Ric(X,V)\to 0$ as $c\to 0$. Hence, if we define $t_0$ such that $f_{a,C}'(t_0)=\lambda$, we have positive Ricci curvature on $[0,t_0]$ for $a$ and $c$ sufficiently small, i.e.\ for $a$ and $r_0$ sufficiently small, where $r_0$ depends on $p,q,\lambda,\hat{g}$ and $A$.
		
		We now rescale the metric $g_{f,h}$ by $\frac{c}{a}$, that is, we replace $f$, $h$ and $t_0$ by $\frac{c}{a}f(\frac{a}{c}\cdot)$, $\frac{c}{a}h(\frac{a}{c}\cdot)$ and $\frac{c}{a}t_0$, respectively, so that
		\[ f(t)=f_C(\tfrac{a}{c}t),\quad h(t)=\frac{r_0}{h_0(0)}h_0(\tfrac{a}{c}t). \]
		Hence, the induced metric is given by $g_{r_0}$ on $\{0\}\times E$ and by $R^2g_{r_1}$ on $\{1\}\times E$, where $R=f_C(t_0)$ and
		\[ r_1=\frac{r_0}{h_0(0)}\frac{h_0(t_0)}{f_C(t_0)}. \]
		Hence, once $r_0$ is fixed, we can realize any sufficiently small value of $r_1$ by choosing $a$ smaller, since $t_0\to\infty$ as $a\to 0$ by (3) and $\frac{h_0(t)}{f_C(t)}\to0$ as $t\to\infty$.
		
		For the second fundamental form we obtain
		\begin{align*}
			\II_{\{0\}\times E}&=a r_0h_0'(0)\mathrm{pr}_{\mathcal{V}}^*\hat{g}=a\frac{h_0'(0)}{r_0}g_{r_0}|_{\mathcal{V}},\\
			\II_{\{\frac{c}{a}t_0\}\times E}&=\lambda f_C(t_0)\pi^*\check{g}+a\frac{r_0}{h_0(0)}h_0'(t_0)h_0(t_0)\mathrm{pr}_{\mathcal{V}}^*\hat{g}\geq \lambda R\pi^*\check{g}=\lambda R g_{r_1}|_{\mathcal{H}}.
		\end{align*}
		Thus, we can realize any sufficiently small value for the principal curvatures at $t=0$ (note that we need to reverse the sign since the outward normal is given by $-\partial_t$) by choosing $a$ sufficiently small. Hence, we obtain all required boundary conditions. Note that the value of $a$, and therefore of $r_1$ depends on $p,q,\lambda,\hat{g},A,\nu$ and $r_0$.
	\end{proof}
	
	In the case of a one-dimensional fibre the condition $\Ric_{\hat{g}}>0$ in Proposition \ref{P:bdl_warping} is not satisfied. By the construction of Sha--Yang \cite{SY91} we have the following result in this case:
	\begin{proposition}\label{P:bdl_warping_S1}
		Suppose that $\Ric_{\check{g}}\geq (q-1)\check{g}$ with $q\geq 2$ and suppose that $\pi$ is the trivial $S^1$-bundle over $B$, i.e.\ $E=B\times S^1$. Then, for any $\lambda\in(0,1)$ there exists a metric of positive Ricci curvature on $B\times D^2$ such that the induced metric on the boundary is of the form $\check{g}+R^2ds_1^2$ for some $R>0$ and the second fundamental form on the boundary satisfies
		\begin{align*}
			\II_{B\times S^1}(\partial_t,\partial_t)&\geq 0,\\
			\II_{B\times S^1}(X,\partial_t)&=0,\\
			\II_{B\times S^1}(X,X)&\geq \lambda \check{g}(X,X).
		\end{align*}
		Here $\partial_t$ is a unit tangent vector field on $S^1$ and $X$ is tangent to $B$.
	\end{proposition}
	\begin{proof}
		Similarly as in the proof of Proposition \ref{P:bdl_warping} we consider the doubly warped product metric
		\[ g_{f,h}=dt^2+f(t)^2\check{g}+h(t)^2ds_1^2, \]
		$t\in[0,t_0]$. Here we collapse the sphere $S^1$ at $t=0$ so that we obtain a smooth metric on $B\times D^2$, that is, we require that the functions $f$ and $h$ satisfy the boundary conditions \ref{EQ:dw_boundary_0}. We now choose these functions as the warping functions defined in \cite[Lemma 3.3]{Re24}, in particular we have $f'(t_0)=\lambda$ and $h'(t_0)=0$ and, by the formulas for the Ricci curvatures given in the proof of Proposition \ref{P:bdl_warping}, we also have positive Ricci curvature. Hence, after rescaling by $f(t_0)^{-1}$, we obtain a metric that satisfies all required conditions.
	\end{proof}
	
	Finally, we will also consider the second fundamental form on the boundary of the total space of a bundle $E\to B$, where $B$ has non-empty boundary.
	\begin{lemma}\label{L:II_bdl}
		If $B$ has non-empty boundary and $F$ has empty boundary, then the second fundamental form on the boundary $\partial E=\pi^{-1}(\partial B)$ for the metric \eqref{EQ:submersion_metric} is given by
		\[ \II_{\partial E}=\pi^*\II_{\partial B}. \]
	\end{lemma}
	\begin{proof}
		By the Gauss lemma, we can identify a neighbourhood of the boundary $\partial B$ with $\partial B\times [0,\varepsilon)$ and the metric $\check{g}$ takes the form $\check{g}=dt^2+\check{g}_t$ in this neighbourhood. Hence, a neighbourhood of $\partial B$ can be identified with $\pi^{-1}(\partial B)\times[0,\varepsilon)$ and the metric $g$ takes the form
		\[ g=dt^2+\pi^*\check{g}_t+\mathrm{pr}^*_{\mathcal{V}}\hat{g} \]
		in this neighbourhood. The claim now follows from Lemma \ref{L:curv_form}.
		
	\end{proof}

	\subsection{Deforming and gluing metrics of positive Ricci curvature}\label{SS:glue_deform}
	
	In this subsection we collect several results on gluing and deforming Riemannian metrics of positive Ricci curvature that will be helpful throughout Sections \ref{S:build_block} and \ref{S:glued_spaces}.
	
	First, we state Perelman's gluing result \cite{Pe97}.
	
	\begin{proposition}[{\cite{Pe97}, see also \cite[Theorem 2]{BWW19}}]\label{P:gluing}
		Let $(M_1^n,g_1)$ and $(M_2^n,g_2)$ be Riemannian manifolds of positive Ricci curvature and assume that there exists an isometry $\phi\colon \partial M_1\to\partial M_2$ such that $\II_{\partial M_1}+\phi^*\II_{\partial M_2}$ is positive semi-definite. Then the glued manifold $M_1\cup_\phi M_2$ admits a Riemannian metric of positive Ricci curvature that coincides with $g_i$ on $M_i$ outside an arbitrarily small neighbourhood of the gluing area.
	\end{proposition}
	
	In the proof of Theorem \ref{T:NE_core_mtrc} below we will need a refined version of Proposition \ref{P:gluing} to manifolds with corners. For that, let $M^n$ be a manifold with corners of codimension 2, that is, $M$ is locally diffeomorphic to $(\R_{\geq 0})^2\times \R^{n-2}$. This gives a partition of $\partial M$ into \emph{faces}, i.e.\ the closures of connected components of the set corresponding along charts to boundary points in $(\R_{\geq 0})^2\times \R^{n-2}\setminus\{0\}$. The connected components of the boundaries of the faces are the \emph{corners}, which correspond to points in $\{0\}\times \R^{n-2}$ along charts. In particular, every corner is obtained as the connected component of the intersection between two faces. We will assume that these two faces are always distinct.
	
	Given a corner $Z\subseteq M$, we can smooth the corner, which eliminates $Z$ and replaces the faces $Y_1$ and $Y_2$ that intersect at $Z$ by the glued face $Y_1\cup_Z Y_2$. Locally, the smoothing corresponds to identifying $(\R_{\geq 0})^2\times \R^{n-2}$ with $\R_{\geq0}\times \R^{n-1}$.
	
	If $M$ is equipped with a Riemannian metric, we define the angle $\theta$ along the corners as the function assigning to a point of a corner the interior angle between the two faces intersecting at this point.
	\begin{proposition}[{\cite[Theorem D and Corollary D]{Bu20}}]\label{P:gluing_corner}
		Let $(M_1^n,g_1)$ and $(M_2^n,g_2)$ be Riemannian manifolds of positive Ricci curvature with corners of codimension $2$ such that there exists an isometry $\phi\colon Y_1\to Y_2$ between two faces $Y_1\subseteq M_2$ and $Y_2\subseteq M_2$. Suppose that
		\begin{enumerate}
			\item The sum of second fundamental forms $\II_1+\phi^*\II_2$ is positive along $Y_1$,
			\item The sum of angles $\theta_1+\theta_2\circ\phi$ along the corners $\partial Y_1$ is less than $\pi$,
			\item The second fundamental form is positive near $\partial Y_i$ on each face intersecting $Y_i$.
		\end{enumerate}
		Then there exists a Riemannian metric $g$ of positive Ricci curvature on the manifold with corners obtained from $M_1\cup_\phi M_2$ by smoothing the corners $\partial Y_1\cong \partial Y_2$ that coincides with $g_1$ resp.\ $g_2$ outside an arbitrarily small neighbourhood of the gluing area and such that the second fundamental form on all faces intersecting $Y_i$ is positive in this neighbourhood.
		
		Moreover, if $Z\subseteq \partial Y_1$ is a corner with corresponding faces $\tilde{Y}_i\subseteq M_i$ such that $Z\subseteq Y_1\cap\tilde{Y}_1$ and $\phi(Z)\subseteq Y_2\cap\tilde{Y}_2$, such that the metrics on $\tilde{Y}_i$ are given by $ds^2+h_i(s)$ with $h_i''<0$, where $s$ is the distance from $Z$, resp.\ $\phi(Z)$, then the induced metric of $g$ on $\tilde{Y}_1\cup_{\phi|_Z}\tilde{Y}_2$ is of the form $ds^2+h(s)$ with $h''<0$.
	\end{proposition}
	In fact, as seen in \cite[Appendix 2.3]{Bu20}, if the induced metrics on $\tilde{Y}_i$ are warped product metrics, then the induced metric of $g$ on $\tilde{Y}_1\cup_{\phi|_Z}\tilde{Y}_2$ is again a warped product metric.
	
	To satisfy the assumption of having isometric boundaries in Proposition \ref{P:gluing}, we have the following result to modify the metric on the boundary.
	
	\begin{proposition}[{\cite[Theorem 7]{Bu20}}]\label{P:deform}
		Let $g_0$ and $g_1$ be Riemannian metrics on a closed manifold $M^n$ that lie in the same path component in the space of Ricci-positive metrics on $M$. Then, for all $\nu\in(0,1)$ there exist $R>0$ and a Riemannian metric $g$ of positive Ricci curvature on $M\times[0,1]$, such that
		\begin{enumerate}
			\item The induced metric on $\{0\}\times M$ is $g_0$ and the principal curvature are all at least $-\nu$,
			\item The induced metric on $\{1\}\times M$ is $R^2g_1$ for some $R>0$ and the principal curvature are all positive.
		\end{enumerate}
	\end{proposition}
	
	Further, to satisfy the assumption on having positive semi-definite second fundamental form in Proposition~\ref{P:gluing}, we have the following deformation result.
	
	\begin{proposition}[{\cite[Proposition 1.2.11]{Bu19a}}]\label{P:deform_II>=0}
		Let $(M^n,g)$ be a Riemannian manifold of positive Ricci curvature that has positive semi-definite second fundamental form on the boundary. Then there exists a Riemannian metric $g'$ of positive Ricci curvature on $M$ with positive second fundamental form on the boundary such that $g|_{\partial M}=g'|_{\partial M}$.
	\end{proposition}

	In fact, by using the deformation results of Ehrlich \cite{Eh76}, we can generalize Proposition \ref{P:deform_II>=0} as follows.
	
	\begin{proposition}\label{P:deform_II>=0_bdry}
		Let $(M^n,g)$ be a compact Riemannian manifold of non-negative Ricci curvature and suppose that the second fundamental form on the boundary is non-negative. If there exists a point in the interior of $M$ at which all Ricci curvatures are positive, then there exists a smooth family of metrics $g_t$, $t\in[0,1]$, on $M$ such that the following holds:
		\begin{enumerate}
			\item $g_0=g$,
			\item $g_t$ has non-negative Ricci curvature for all $t\in[0,1]$,
			\item $g_t$ induces the same metric as $g$ on $\partial M$ and has non-negative second fundamental form for all $t\in[0,1]$,
			\item $g_1$ has positive Ricci curvature on $M$ and positive second fundamental form on $\partial M$.
		\end{enumerate}
	\end{proposition}
	\begin{proof}
		We consider the metric $g$ in normal coordinates to the boundary, i.e.\ we can identify a tubular neighbourhood of $\partial M$ with $[0,\varepsilon]\times \partial M$ and the metric on this part is of the form $dt^2+h_t$, where $h_t$ is a smooth family of Riemannian metrics on $\partial M$. Then the subset 
		\[M_{\frac{\varepsilon}{4}}=\{x\in M\mid d(p,\partial M)\geq\tfrac{\varepsilon}{4} \}\]
		is compact. By possibly choosing $\varepsilon$ smaller if necessary, we can assume that $M_{\frac{\varepsilon}{4}}$ contains a point at which all Ricci curvatures are positive.
		
		By the deformation results of Ehrlich \cite[Proposition in Section 5]{Eh76}, there is a smooth family of metrics $g_s$, $s\in[0,\delta]$ with $g_0=g$ and for all $s\in(0,\delta]$ the metric $g_s$ has non-negative Ricci curvature, strictly positive Ricci curvature on $M_{\frac{\varepsilon}{4}}$ and coincides with $g$ outside $M_{\frac{\varepsilon}{8}}$. Note that the statement of \cite[Proposition in Section 5]{Eh76} requires a compact set in a complete non-compact manifold of non-negative Ricci curvature. However, only the metric in a neighbourhood of the compact subset is relevant for the proof, so that we can apply this result in our situation as well.
		
		We again write $g_s$ in normal coordinates in a tubular neighbourhood of $\partial M$ as $dt^2+h_t^s$, and for $s$ sufficiently small, we can assume that $g_s$ has this form on the tube $[0,\frac{3\varepsilon}{4}]\times \partial M$, cf.\ \cite{Eh74}, and that $g_s$ has positive Ricci curvature for $t\geq \frac{3\varepsilon}{8}$. We fix such $s$ and write again $g$ for $g_s$ and $\varepsilon$ for $\frac{3}{4}\varepsilon$, i.e.\ we have that $g=dt^2+h_t$ on $[0,\varepsilon]\times \partial M$ and $g$ has positive Ricci curvature for $t\geq \frac{\varepsilon}{2}$.
		
		To modify $g$ near the boundary let $C>0$ and let $f\colon[0,\varepsilon]\to\R$ be a smooth function satisfying the following:
		\begin{enumerate}
			\item $f(t)=1-e^{Ct}$ for $t\in[0,\frac{\varepsilon}{2}]$,
			\item $f$ and all its derivatives vanish at $t=\varepsilon$.
		\end{enumerate}
		Then, for $a\geq0$ sufficiently small, we define a smooth metric $g_{a,C}$ on $M$ by
		\[
		g_{a,C}=\begin{cases}
			dt^2+(1+af(t))^2h_t,\quad & \text{on }[0,\varepsilon]\times \partial M,\\
			g,\quad &\text{else.}
		\end{cases}
		\]
		Note that $g_{0,C}=g$ and $a\mapsto g_{a,C}$ is a smooth deformation of $g$ that coincides with $g$ on $\partial M$. We will now show that for $C$ sufficiently large and $a$ sufficiently small, the metric $g_{a,C}$ has strictly positive Ricci curvature on $M$ and strictly positive second fundamental form on $\partial M$.
		
		We set $\tilde{h}_t=(1+af(t))^2h_t$ and calculate
		\begin{align*}
			\tilde{h}_t'&=2af'(1+af)h_t+(1+af)^2h_t'=(1+af)^2\left( 2a\frac{f'}{1+af}h_t+h_t' \right),\\
			\tilde{h}_t''&=2a(f''(1+af)+a{f'}^2)h_t+4af'(1+af)h_t'+(1+af)^2h_t''\\
			&=(1+af)^2\left( 2a\frac{f''}{1+af}h_t+2a^2\frac{{f'}^2}{(1+af)^2}h_t+4a\frac{f'}{1+af}h_t'+h_t'' \right).
		\end{align*}
		Since $f'(0)<0$, it follows from Lemma \ref{L:curv_form} that the second fundamental form at the boundary is positive for all $a$ and $C$ (note that the outward unit normal is $-\partial_t$ in this case).
		
		Further, by Lemma \ref{L:curv_form}, there exists a constant $c>0$ that depends on $h_t$, $h_t'$ and $h_t''$ such that the Ricci curvatures of $g_{a,C}$ can be estimated as follows (note that the Ricci tensor is invariant under scaling):
		\begin{align*}
			\Ric(\partial_t,\partial_t)&\geq -a(n-1)\frac{f''}{1+af}-a\frac{|f'|}{1+af}c+\Ric^g(\partial_t,\partial_t),\\
			|\Ric(\tilde{v},\partial_t)-\Ric^g(\tilde{v},\partial_t)|&\leq a\frac{|f'|}{1+af}c,\\
			\Ric(\tilde{v},\tilde{v})&\geq -a\frac{f''}{1+af}-c\left(\left| 1-\frac{1}{(1+af)^2} \right|+ a\frac{|f'|}{1+af}\right)+(1+af)^2\Ric^g(\tilde{v},\tilde{v}),
		\end{align*}
		where $\tilde{v}$ is a unit tangent vector of $\partial M$ with respect to $\tilde{h}_t$. Hence, for any unit tangent vector $\tilde{u}$ of $g_{a,C}$, we can estimate
		\[ \Ric(\tilde{u},\tilde{u})\geq \Ric^g(\tilde{u},\tilde{u})+a\left( \frac{-f''}{1+af}-c\left( \frac{1}{a}\left| 1-\frac{1}{(1+af)^2} \right|+\frac{|f'|}{1+af} \right) \right). \]
		
		Since $g$ has positive Ricci curvature for $t\geq \frac{\varepsilon}{2}$, it follows that $g_{a,C}$ also has positive Ricci curvature for $t\geq \frac{\varepsilon}{2}$ for fixed $C$ and for $a$ sufficiently small.
		
		For $t\leq \frac{\varepsilon}{2}$ we have $f(t)=1-e^{Ct}$, so
		\[ C|f'|=-f'' \]
		and since $|f(t)|=e^{Ct}-1\leq e^{Ct}= \frac{-f''(t)}{C^2}$, we obtain
		\[ \left| 1-\frac{1}{(1+af)^2} \right|=a|f|\left| \frac{2+af}{(1+af)^2} \right|\leq -3a\frac{f''}{C^2(1+af)} \]
		for $a$ so small that $af\leq 1$. Hence, using $\Ric^g\geq 0$, we obtain
		\[ \Ric(\tilde{u},\tilde{u})\geq a\frac{-f''}{1+af}\left( 1-c\left( \frac{3}{C^2}+\frac{1}{C} \right) \right). \]
		For $C$ sufficiently large, this expression is positive, showing that the Ricci curvatures are strictly positive for $t\in[0,\frac{\varepsilon}{2}]$.
	\end{proof}
	
	In a similar way, we can prove the following generalisation of Proposition \ref{P:gluing}.
	
	\begin{proposition}\label{P:gluing_Ric>=0}
		In the situation of Proposition \ref{P:gluing} suppose that $g_1$ and $g_2$ merely have non-negative Ricci curvature and $M_1$ contains a point in its interior where all Ricci curvatures are strictly positive. Then the glued manifold $M_1\cup_{\phi}M_2$ admits a Riemannian metric of positive Ricci curvature.
	\end{proposition}
	\begin{proof}
		We use a similar deformation as in Proposition \ref{P:deform_II>=0_bdry} applied to $g_1$ so that the sum $\II_{\partial M_1}+\phi^*\II_{\partial M_2}$ is strictly positive. Similarly as in the proof of Proposition \ref{P:deform_II>=0_bdry} then replace the metric $g_2$ by the metric $g_{a,C}$ for $a,C>0$ on $M_2$ defined by
		\[ g_{a,C}=\begin{cases}
			dt^2+(1+af(t))^2h_t,\quad & \text{on }[0,\varepsilon]\times \partial M_2,\\
			(1+af(\varepsilon))^2 g_2,\quad & \text{else,}
		\end{cases} \]
		where we have written $g_2$ near a tubular neighbourhood $[0,\varepsilon]\times\partial M_2$ of $\partial M_2$ as $dt^2+h_t$ and $f$ is a smooth function satisfying $f(\varepsilon)=0$, $f''<0$, all derivatives of $f$ vanish at $t=\varepsilon$ and
		\[ -f''\geq C f'. \]
		
		Similar arguments as in the proof of Proposition \ref{P:deform_II>=0_bdry} show that the resulting metric has non-negative Ricci curvature and strictly positive Ricci curvature on $[0,\varepsilon]\times \partial M_2$ for $a$ sufficiently small and $C$ sufficiently large, while the induced metric on the boundary only changes by a scalar multiple and the second fundamental form converges to that of $g_2$ as $a\to 0$. Hence, we can glue this metric to the modified metric $g_1$ for $a>0$ sufficiently small and obtain positive Ricci curvature by Proposition \ref{P:gluing} in a neighbourhood of the gluing area, and non-negative Ricci curvature globally. By the deformation in \cite{Eh76} we can deform this metric into a metric of positive Ricci curvature.
	\end{proof}
	
	We will also need the following deformation result, that allows to combine two given metrics on a manifold.
	
	\begin{proposition}[{\cite[Theorem 1.10]{Wr02}, see also \cite[Theorem 1.2]{BH22}}]\label{P:local_flex}
		Let $(M^n,g)$ be a Riemannian manifold of positive Ricci curvature and let $N\subseteq M$ be a closed embedded submanifold. Let $U\subseteq M$ be an open neighbourhood of $N$ and let $g'$ be a Riemannian metric of positive Ricci curvature on $U$. If the $1$-jets of $g$ and $g'$ coincide on $N$, then there exists a smooth family $g_t$ of Ricci-positive metrics on $M$ such that
		\begin{enumerate}
			\item $g_0=g$,
			\item $g_1$ coincides with $g$ on $M\setminus U$ and with $g'$ in a neighbourhood of $N$.
		\end{enumerate}
	\end{proposition}


	To finish this section, we show that the definitions of core metrics given in Definition \ref{D:core}, \cite{Bu19} and \cite{Re22a} are all equivalent.
	\begin{lemma}\label{L:core_equ}
		Let $M^n$ be a manifold. Then the following are equivalent:
		\begin{enumerate}
			\item $M$ admits a core metric (Definition \ref{D:core}).
			\item $M$ admits a Riemannian metric of positive Ricci curvature and an embedding $D^n\subseteq M$ such that the induced metric on the boundary $\partial(M\setminus {D^n}^\circ)$ is round and has positive second fundamental form (Definition in \cite{Bu19}).
			\item $M$ admits a Riemannian metric of positive Ricci curvature and an embedding $D^n\subseteq M$ such that the induced metric on the boundary $\partial(M\setminus {D^n}^\circ)$ is round and has non-negative second fundamental form (Definition in \cite{Re22a}).
		\end{enumerate}
	\end{lemma}
	\begin{proof}
		By Proposition \ref{P:deform_II>=0}, items (2) and (3) are equivalent. Further, (1) implies (3) since a core metric satisfies the requirements of the metric in (3). Finally, suppose that (2) holds. Then, using Proposition \ref{P:gluing} we can glue a geodesic ball in $S^n$ which is slightly bigger than a hemisphere to $\partial(M\setminus{D^n}^\circ)$, which results in a Riemannian metric of positive Ricci curvature on $M$ with an embedded round hemisphere, i.e.\ a core metric.
	\end{proof}
	
	\section{Construction of the main building block}\label{S:build_block}
	
	In this section we construct a specific metric on $[0,1]\times M$, where $M^n$ is a manifold admitting a core metric. This \textquotedblleft building block\textquotedblright\ will be the main ingredient in the proof of Theorem \ref{T:core_bdl}.
	
	The goal is to introduce a corner at $\{1\}\times S^{n-1}$ in the boundary component $\{1\}\times M$ of angle less than $\frac{\pi}{2}$, while having positive Ricci curvature on $[0,1]\times M$ and positive second fundamental form on $\{1\}\times M$. 
	The metric on $[0,1]\times M$ will, roughly speaking, be close to a cone metric from which we cut out a piece of $[0,1]\times D^n$ along a hypersurface. This hypersurface is tangent to a slice $\{t\}\times M$ near the centre of $\{t\}\times D^n\subseteq \{t\}\times M$, and almost vertical at $\{1\}\times S^{n-1}$, thus creating the desired angle. To keep the second fundamental form of the resulting boundary component positive, the slope of this hypersurface cannot increase to quickly in radial directions of $D^n$. For this to be possible it will be essential to first \textquotedblleft elongate\textquotedblright\ the embedded disc $D^n\subseteq M$ as follows:
	
	\begin{lemma}\label{L:cone_metric}
		Let $M^n$ be a manifold that admits a core metric $g$ and let $K\in(0,1)$ such that $\Ric_g\geq K^2(n-1)g$. Then, for any $\varepsilon_1,\varepsilon_2>0$ sufficiently small, there exists a Riemannian metric $\cone{g}$ on $M$ with $\Ric_{\cone{g}}\geq (n-1)\cone{g}$ and an embedding $D^n\subseteq M$, such that the following holds:
		\begin{enumerate}
			\item There is an annulus $A\subseteq D^n$, which we identify with $[\varepsilon_1,\frac{\pi}{2}+\varepsilon_2]\times S^{n-1}$, on which the metric $\cone{g}$ is given by
			\[  ds^2+K^2\sin^2(s)ds_{n-1}^2 \]
			for all $s\in[\varepsilon_1,\frac{\pi}{2}+\varepsilon_2]$,
			\item The metrics $\cone{g}$ and $g$ lie in the same path component of the space of Ricci-positive metrics on $M$.
		\end{enumerate}
	\end{lemma}
	\begin{proof}
		Since $g$ is a core metric, there exists an embedded disc $D^n\subseteq M$ on which $g$ is given by $ds^2+\sin^2(s)$, where $s\in[0,\frac{\pi}{2}]$. By applying Proposition \ref{P:local_flex} to the boundary of $M\setminus{D^n}^\circ$, which is a round and totally geodesic sphere $S^{n-1}$, we can in fact assume that $g$ has this form for $s\in[0,\frac{\pi}{2}+\varepsilon_2']$ for some sufficiently small $\varepsilon_2'>0$.
		
		For $t\in[0,1]$ we set $K_t=1-t(1-K)$ and define the function $f_t\colon[(1-K)t\frac{\varepsilon_1}{2K},\frac{\pi}{2K_t}+\frac{\varepsilon_2'}{2}(1+K_t^{-1})]\to[0,1]$ by
		\[ f_t(s)=\begin{cases}
			\sin(s-(1-K)\frac{\varepsilon_1}{2K}t),\quad & s\in[(1-K)t\frac{\varepsilon_1}{2K},\frac{\varepsilon_1}{2K}], \\
			\sin(sK_t),\quad & s\in[\frac{\varepsilon_1}{2K},\frac{\pi+\varepsilon_2'}{2}K_t^{-1} ],\\
			\sin(s-\frac{\pi+\varepsilon_2'}{2}(K_t^{-1}-1)),\quad & s\in[\frac{\pi+\varepsilon_2'}{2}K_t^{-1},\frac{\pi}{2K_t}+\frac{\varepsilon_2'}{2}(1+K_t^{-1})].
		\end{cases} \]
		Note that $f_t$ is continuous, and smooth except at $s_1=\frac{\varepsilon_1}{2K}$ and $s_2=\frac{\pi+\varepsilon_2'}{2}K_t^{-1}$ with ${f_t'}_-(s_i)>{f_t'}_+(s_i)$ at these points. We smooth $f_t$ near these points as follows:
		
		For $s_1=\frac{\varepsilon_1}{2K}$ and $\delta>0$ small we replace $f_t$ on $[s_1-\delta,s_1+\delta]$ by $\sin(\chi_\delta)$, where $\chi_\delta\colon[s_1-\delta,s_1+\delta]\to\R$ is a smooth function with $\chi_\delta(s)=s-(1-K)\frac{\varepsilon_1}{2K}t$ near $s=s_1-\delta$, $\chi_\delta=sK_t$ near $s=s_1+\delta$, $\chi_\delta'\in[K_t,1]$ and $\chi_\delta''<0$. We also assume that $\chi_\delta$ varies smoothly in $t$. Similarly, we smooth $f_t$ around $s_2=\frac{\pi+\varepsilon_2'}{2}K_t^{-1}$.
		
		Hence, we constructed a family of smooth functions $f_t$ that vary smoothly in $t$ and a computation shows that $f_t$ satisfies
		\[ \frac{f_t''}{f_t}\leq -K_t^2\quad\text{and}\quad \frac{1-{f_t}^2}{f_t}\geq K_t^2. \]
		By construction, the metric $ds^2+f_t(s)^2 ds_{n-1}^2$ for $s\in[(1-K)t\frac{\varepsilon}{2},\frac{\pi}{2K_t}+\frac{\varepsilon_2'}{2}(1+K_t^{-1})]$ on $D^n$ now glues smoothly with the metric $g$ on $M\setminus D^n$, so that we obtain the desired isotopy of metrics and property (1) for $t=1$ by setting $\varepsilon_2=\frac{\varepsilon_2'}{2}(1-\delta)$ and rescaling by $K$.
		
		Finally, for a unit vector $X$ in $(S^{n-1},ds_{n-1}^2)$, the Ricci curvatures of the metric $ds^2+f_t(s)^2ds_{n-1}^2$ (e.g.\ using Lemma \ref{L:doubly_warped_curv}) are given by
		\begin{align*}
			\Ric(\partial_s,\partial_s)&=-(n-1)\frac{f_t''}{f_t}\geq (n-1)K_t^2,\\
			\Ric(X,\partial_s)&=0,\\
			\Ric(\tfrac{X}{f_t},\tfrac{X}{f_t})&=-\frac{f_t''}{f_t}+(n-2)\frac{1-{f_t'}^2}{f_t^2}\geq (n-1)K_t^2.
		\end{align*}
		Hence for $t=1$, we after rescaling by $K$, we obtain the lower Ricci curvature bound as required.
	\end{proof}
	
	Using this modified core metric, we will now construct the main building block.
	
	\begin{proposition}\label{P:handle1}
		Let $(M^n,g)$ be a Riemannian manifold as in Lemma \ref{L:cone_metric} and for $\varepsilon_1,\varepsilon_2>0$ sufficiently small consider the metric $\cone{g}$ on $M$. We consider $[0,1]\times M$ as a manifold with a corner along $\{1\}\times S^{n-1}$, dividing the boundary component $\{1\}\times M$ into the two faces $\{1\}\times D^n$ and $\{1\}\times (M\setminus {D^n}^\circ)$. Then there exists $\lambda\in(0,1)$ and a Riemannian metric $h$ of positive Ricci curvature on $M\times [0,1]$ such that the following holds:
		\begin{enumerate}
			\item $h$ induces the metric $\cone{g}$ on $\{0\}\times M$ with principal curvatures all at least $-\lambda$,
			\item $h$ induces a metric of positive sectional curvature of the form $ds^2+A(s)^2ds_{n-1}^2$ on the face $\{1\}\times D^n$,
			\item $h$ induces a metric of positive Ricci curvature on the face $\{1\}\times M\setminus{D^n}^\circ$, which is given by $ds^2+B(s)^2ds_{n-1}$ near the boundary $\{1\}\times S^{n-1}$ with positive  second fundamental form on the boundary,
			\item The angles at the corner $\{1\}\times S^{n-1}$ are strictly less than $\frac{\pi}{2}$,
			\item The boundary component $\{1\}\times M$ is convex.
		\end{enumerate}
	\end{proposition}
	Figure \ref{F:handle1} contains a sketch of the metric constructed in Proposition \ref{P:handle1}.
	\begin{figure}
		\includegraphics[scale=0.4]{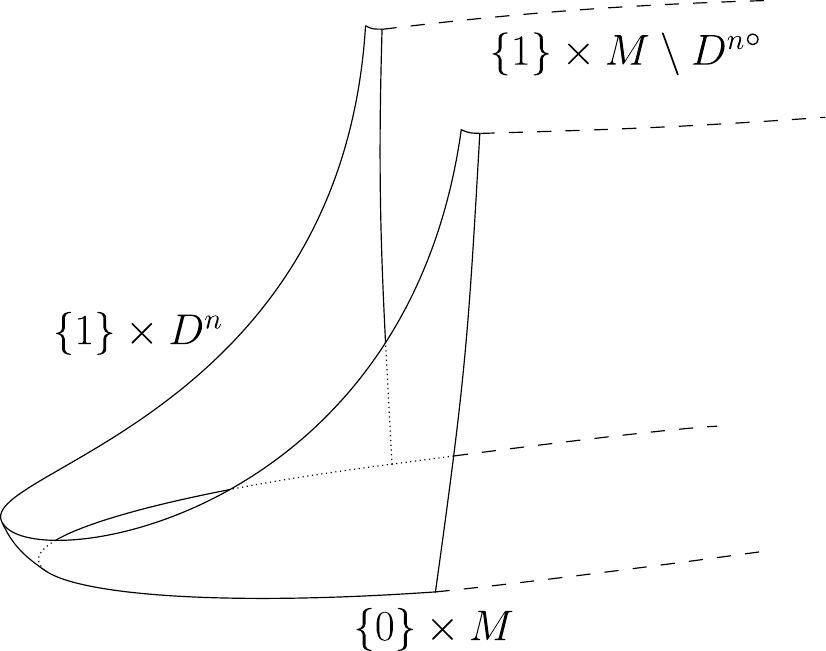}
		\caption{Sketch of the metric constructed in Proposition \ref{P:handle1}.}\label{F:handle1}
	\end{figure}
	\begin{proof}
		We consider a metric of the form
		\[ \overline{g}=dt^2+f(t)^2\cone{g} \]
		on $[-\delta,\infty)\times M$, where $f$ is the function obtained in Lemma \ref{L:fH} for some $0<\lambda_1<\lambda_2<1$, i.e.\ we have
		\begin{enumerate}
			\item $f(0)=1$, $f'(0)=\lambda_2$,
			\item $f''<0$,
			\item $f'>\lambda_1$ and $\frac{f'(t)}{f(t)}>\frac{\lambda_1}{1+\lambda_1 t}$.
		\end{enumerate}
		By choosing $\delta$ smaller if necessary, we can assume that $f'(-\delta)<1$ and we set $\lambda=f'(-\delta)$, so that, after rescaling, the metric on the boundary component $\{-\delta\}\times M$ is given by $\cone{g}$ with principal curvatures given by $-\lambda$. Further, since $f''<0$ and $f'\in(0,1)$, the metric $\overline{g}$ has positive Ricci curvature.
		
		For the other boundary component we will define $\rho\colon M\to[0,\infty)$ and consider the manifold
		\[ M_\rho=\{ (t,x)\in [-\delta,\infty)\times M\mid t\leq \rho(x) \}. \]
		Recall from Lemma \ref{L:cone_metric} that there exists an annulus $A\subseteq D^n\subseteq M$, which we identify with $[\varepsilon_1,\frac{\pi}{2}+\varepsilon_2]\times S^{n-1}$, on which $\cone{g}$ is given by $ds^2+K^2\sin^2(s)ds_{n-1}$. On this piece we define $\rho$ by $\rho(s,x)=\alpha(s)$, where, for some $\varepsilon>0$ and for $\lambda_1$ sufficiently close to $1$ and $\varepsilon_1$ sufficiently small, the function $\alpha$ is given by
		\[ \alpha(s)=\begin{cases}
			\frac{1}{\lambda_1}\left( \frac{1}{\cos(\lambda_1(s-\varepsilon_1))}-1 \right),\quad & s\in[\varepsilon_1,\frac{\pi}{2}],\\
			\alpha(\frac{\pi}{2})+\beta(s),\quad & s\in[\frac{\pi}{2},\frac{\pi}{2}+\varepsilon_2].
		\end{cases} \]
		Here $\beta\colon[\frac{\pi}{2},\frac{\pi}{2}+\varepsilon_2]\to[-1,0]$ is a smooth function satisfying $\beta(\frac{\pi}{2})=0$, $\beta'(\frac{\pi}{2})=-2$, and all derivatives of $\beta$ vanish at $s=\frac{\pi}{2}+\varepsilon_2$.
		
		Further, we define $\rho\equiv 0$ on the piece of $D^n$ bounded by $A$, and $\rho\equiv \alpha(\frac{\pi}{2})+\beta(\frac{\pi}{2}+\varepsilon_2)$ on $M\setminus {D^n}^\circ$. For $\rho$ to be smooth at $s=\varepsilon_1$ we modify $\alpha$ in a small neighbourhood of $s=\varepsilon_1$, for example as in \cite[Lemma 3.1]{RW23a} (and note that $\alpha'(\varepsilon_1)=0$), so that all its derivatives vanish at $s=\varepsilon_1$ and the second derivative in this neighbourhood lies in the interval $[-\delta',\alpha''(0)+\delta']$ for an arbitrarily small $\delta'>0$. For $\delta'$ sufficiently small the subsequent computations will then also hold for the smoothed function.
		
		We will now show that for suitable choices of $\lambda_1,\lambda_2$ and $\varepsilon_1$ the manifold $(M_\rho,\overline{g})$ satisfies the required properties. For that, first note that outside of the piece $\{ \varepsilon_1\leq s\leq\frac{\pi}{2}+\varepsilon_2 \}$ the required properties directly follow from the fact that $\rho$ is locally constant and $f'>0$. Thus, we are left with the case $s\in[\varepsilon_1,\frac{\pi}{2}+\varepsilon_2]$.
		
		We will first consider the second fundamental form. For the case $s\in[\varepsilon_1,\frac{\pi}{2}]$ we obtain from Lemma \ref{L:II_graph_warped} the following:
		\begin{align*}
			\II(\partial_s+\alpha'\partial_t,\partial_s+\alpha'\partial_t)&>\frac{1}{\sqrt{1+\frac{{\alpha'}^2}{f^2}}}\left( \lambda_1(1+\lambda_1\alpha)+2\frac{\lambda_1}{1+\lambda_1\alpha}{\alpha'}^2-\alpha'' \right),\\
			\II(\partial_s+\alpha'\partial_t,X)&=0,\\
			\II(X,X)&>\frac{\cone{g}(X,X)}{\sqrt{1+\frac{{\alpha'}^2}{f^2}}}\left( \lambda_1(1+\lambda_1\alpha)-\frac{\cos(s)}{\sin(s)}\alpha' \right).
		\end{align*}
		We have
		\begin{align*}
			\alpha'(s)&=\frac{\sin(\lambda_1(s-\varepsilon_1))}{\cos^2(\lambda_1(s-\varepsilon_1))},\\
			\alpha''(s)&=\lambda_1\frac{1+\sin^2(\lambda_1(s-\varepsilon_1))}{\cos^3(\lambda_1(s-\varepsilon_1))},
		\end{align*}
		which shows that $\II(\partial_s+\alpha'\partial_t,\partial_s+\alpha'\partial_t)$ is strictly positive, while $\II(X,X)$ is bounded from below by
		\[ \II(X,X)>\frac{\cone{g}(X,X)\cot(s)}{\cos(\lambda_1(s-\varepsilon_1))\sqrt{1+\frac{{\alpha'}^2}{f^2}}}\left( \lambda_1\tan(s)-\tan(\lambda_1(s-\varepsilon_1)) \right).  \]
		We can use the convexity and monotonicity of the tangent function to estimate
		\[  \tan(\lambda_1(s-\varepsilon_1))<\tan(\lambda_1 s)<\lambda_1\tan(s). \]
		Hence, $\II(X,X)>0$, so the second fundamental form is strictly positive.
		
		For $s\in[\frac{\pi}{2},\frac{\pi}{2}+\varepsilon_2]$ we again use Lemma \ref{L:II_graph_warped}. Note that $\alpha(\frac{\pi}{2})\to\infty$ as $\varepsilon_1\to 0$ and $\lambda_1\to 1$, while $\alpha'$ and $\alpha''$ remain constant. In particular, $f(\alpha(s))\to\infty$ as $\varepsilon_1\to 0$ and $\lambda_1\to 1$ for all $s\in[\frac{\pi}{2},\frac{\pi}{2}+\varepsilon_2]$. It follows that $\II>0$ for $\varepsilon_1$ sufficiently small and $\lambda_1$ sufficiently close to $1$.
		
		Next, we consider the induced metric on the boundary $\partial M_\rho$. For the calculations we replace $f$ by the simpler function $t\mapsto \lambda_1 t+1$. It follows from the construction in Lemma \ref{L:fH} that $f$ smoothly converges to this function on compact subsets as $\lambda_2\to\lambda_1$. Hence, if the induced metric on the boundary has positive curvature, resp.\ positive Ricci curvature for this function, the same is true for $\lambda_2$ sufficiently close to $\lambda_1$.
		
		We first assume $s\in [\varepsilon_1,\frac{\pi}{2}]$. For any fixed $x\in S^{n-1}$ consider the path $\gamma$ in $\partial M_\rho$ defined by $\gamma(r)=(\alpha(s(r)),s(r),x)$, where $s(r)$ is defined by
		\[ s(r)=\frac{\arctan(\lambda_1 r)}{\lambda_1}+\varepsilon_1 \]
		and $r$ is such that $s(r)\in[\varepsilon_1, \frac{\pi}{2}]$. Then $\gamma'(r)$ is a unit vector for the metric $dt^2+(1+\lambda_1 t)^2\cone{g}$, hence, we can express the metric on the piece $\varepsilon_1\leq s\leq \frac{\pi}{2}$ as
		\[ dr^2+(1+\lambda_1\alpha(s(r)))^2K^2\sin^2(s(r))ds_{n-1}^2. \]
		We set
		\[\varphi(r)=(1+\lambda_1\alpha(s(r)))\sin(s(r))=\sin\left(\frac{\arctan(\lambda_1 r)}{\lambda_1}+\varepsilon_1\right)\sqrt{1+\lambda_1^2r^2}. \]
		A calculation shows that the first two derivatives of $\varphi$ are given as follows:
		\begin{align*}
			\varphi'(r)&=\frac{\cos\left(\frac{\arctan(\lambda_1 r)}{\lambda_1}+\varepsilon_1\right)+\lambda_1^2 r \sin\left(\frac{\arctan(\lambda_1 r)}{\lambda_1}+\varepsilon_1\right)}{\sqrt{1+\lambda_1^2r^2}},\\
			\varphi''(r)&=\frac{\sin\left(\frac{\arctan(\lambda_1 r)}{\lambda_1}+\varepsilon_1\right)}{(1+\lambda_1^2r^2)^{\frac{3}{2}}}(\lambda_1^2-1).
		\end{align*}
		We have $\varphi'(0)=\cos(\varepsilon_1)<1$ and $\varphi''<0$. Hence, all sectional curvatures are positive.
		
		For $s\in[\frac{\pi}{2},\frac{\pi}{2}+\varepsilon_2]$ the metric on the boundary is given by
		\[ g_1=(\alpha'(s)^2+f(\alpha(s))^2)ds^2+f(\alpha(s))^2K^2\sin^2(s)ds_{n-1}^2, \]
		which we can write as
		\[ g_1=f(\alpha(s))^2\left( \left(\frac{\beta'(s)^2}{f(\alpha(s))^2}+1\right)ds^2+K^2\sin^2(s)ds_{n-1}^2 \right). \]
		Since $\alpha(s)\in[\alpha(\frac{\pi}{2})-1,\alpha(\frac{\pi}{2})]$ and $\alpha(\frac{\pi}{2})\to\infty$ as $\varepsilon_1\to0$ and $\lambda_1\to 1$, we obtain $\frac{f(\alpha(s))}{f(\alpha(\frac{\pi}{2}))}\to 1$ as $\varepsilon_1\to0$ and $\lambda_1\to 1$. Hence, after rescaling by $f(\alpha(\frac{\pi}{2}))^{-1}$, this metric smoothly converges to $dr^2+K^2\sin^2(r)ds_{n-1}^2$. Thus, for $\varepsilon_1$ sufficiently small and $\lambda_1$ sufficiently close to $1$, the sectional curvature is positive. Moreover, we obtain from the Koszul formula for the second fundamental form at $s=\frac{\pi}{2}$ for a unit tangent vector $X\in TS^{n-1}$:
		\begin{align*}
			\II_{S^{n-1}}(X,X)&=g_1\left(\nabla^{g_1}_X\left(-\frac{1}{\sqrt{\beta'(\frac{\pi}{2})^2+f(\alpha(\frac{\pi}{2}))^2}}\partial_s \right),X\right)\\
			&=-\frac{f(\alpha(\tfrac{\pi}{2}))f'(\alpha(\tfrac{\pi}{2}))\beta'(\tfrac{\pi}{2})}{\sqrt{\beta'(\frac{\pi}{2})^2+f(\alpha(\frac{\pi}{2}))^2}}>0 \\
		\end{align*}
		
		
		Finally, it remains to determine the angle $\theta$ at the corner $s=\frac{\pi}{2}$, i.e.\ we need to determine the angle between the tangent vectors $\alpha'_-(\frac{\pi}{2})\partial_t+\partial_s$ and $-\beta'(\frac{\pi}{2})\partial_t-\partial_s$. Again by replacing $f$ by $t\mapsto \lambda_1 t+1$, we obtain
		\begin{align*}
			\cos(\theta)&=-\frac{\alpha_-'(\frac{\pi}{2})\beta'(\frac{\pi}{2})+f(\alpha(\frac{\pi}{2}))^2}{\sqrt{\alpha_-'(\frac{\pi}{2})^2+f(\alpha(\frac{\pi}{2}))^2}\sqrt{\beta'(\frac{\pi}{2})^2+f(\alpha(\frac{\pi}{2}))^2}},
		\end{align*}
		which is positive if and only if
		\begin{align*}
			0&<-\alpha_-'(\tfrac{\pi}{2})\beta'(\tfrac{\pi}{2})-f(\alpha(\tfrac{\pi}{2}))^2\\
			&=2\frac{\sin(\lambda_1(\frac{\pi}{2}-\varepsilon_1))}{\cos^2(\lambda_1(\frac{\pi}{2}-\varepsilon_1))}-\frac{1}{\cos^2(\lambda_1(\frac{\pi}{2}-\varepsilon_1))}.
		\end{align*}
		We have $\lambda(\frac{\pi}{2}-\varepsilon_1)\to\frac{\pi}{2}$ as $\varepsilon_1\to 0$ and $\lambda_1\to 1$, so that $\cos(\theta)>0$ for $\varepsilon_1$ sufficiently small and $\lambda_1$ sufficiently close to $1$, which implies $\theta<\frac{\pi}{2}$.

	\end{proof}

	\begin{proposition}\label{P:handle2}
		Let $(M^n,g)$ be a compact Riemannian manifold of positive Ricci curvature with convex boundary $\partial M\cong S^{n-1}$ and suppose that near the boundary the metric is of the form $ds^2+B(s)^2ds_{n-1}^2$, where $s$ denotes the distance to the boundary. Then, for any $\varepsilon,\nu>0$ sufficiently small, there exists a metric of positive Ricci curvature on $[0,1]\times M$ such that the following holds:
		\begin{enumerate}
			\item The induced metric on the face $\{0\}\times M$ is given by $g$ and the principal curvatures are all at least $-\nu$,
			\item The induced metric on the face $[0,1]\times S^{n-1}$ is of the form $dt^2+A(t)^2ds_{n-1}$ and has positive sectional curvature, and the second fundamental form is non-negative and strictly positive near $t=0$.
			\item The angles at the corner $\{0\}\times S^{n-1}$ are all at most $\frac{\pi}{2}+\varepsilon$.
			\item The induced metric $g_1$ on the face $\{1\}\times M$ has positive Ricci curvature and the metric near this face is of the form $dt^2+f(t)^2 g_1$, and all derivatives of $f$ vanish at $t=1$.
		\end{enumerate}
	\end{proposition}
	\begin{proof}
		First note that the metric $dt^2+f(t)^2g$ on $[0,1]\times M$ with a suitable concave function $f$ already satisfies all required properties except that it does not have strictly positive second fundamental form on $[0,1]\times S^{n-1}$ near $t=0$. To achieve this we slightly modify the metric near $t=0$ at the cost of creating an angle strictly bigger than $\frac{\pi}{2}$ at $t=0$.
		
		By assumption the metric $g$ near $S^{n-1}$ is of the form $ds^2+B(s)^2ds_{n-1}^2$ with $s\in[0,\delta)$ for some $\delta>0$. Since $g$ has positive Ricci curvature, we have $B''<0$, and since the second fundamental form at $s=0$ is positive, we have $B'(0)<0$. We slightly extend $B$ to a bigger interval $(-\delta',\delta)$ while preserving $B'<0$ and $B''<0$.
		
		Now let $f\colon[0,\infty)\to(0,\infty)$ be the function obtained in Lemma \ref{L:fH} for some $\lambda_1,\lambda_2\in(0,1)$ and modify it in a small neighbourhood of some large $t_0>0$ to have vanishing derivatives at $t=t_0$ while having $f''(t)<0$ on $[0,t_0)$. We consider the metric $\overline{g}=dt^2+f(t)^2 g$ on $[0,t_0]\times M$. For $\lambda_1,\lambda_2$ sufficiently small, the metric $\overline{g}$ has positive Ricci curvature and satisfies (1) at $t=0$.
		
		Next, for a smooth function $\beta\colon[0,t_0]\to (-\delta',0]$ with $\beta(0)=0$ we consider the manifold
		\[ M'=\{ (t,s,x)\in [0,t_0]\times (-\delta',0]\times S^{n-1}\mid \beta(t)\leq s \}\subseteq [0,t_0]\times (-\delta',0]\times S^{n-1} \]
		and equip it with the metric $dt^2+f(t)^2(ds^2+B(s)^2ds_{n-1}^2)$. Note that it glues smoothly with $([0,t_0]\times M,\overline{g})$ along $[0,t_0]\times S^{n-1}$. We claim that for a suitable choice of $\beta$, this glued manifold satisfies all required properties after reparametrizing the interval $[0,t_0]$ to $[0,1]$.
		
		Let $a>0$ and $b>1$, and define
		\[\beta'(t)=a\left(\frac{t}{b}-1\right)\chi\left(t-b\right) \]
		with $\beta(0)=0$. Here $\chi\colon\R\to\R$ is a smooth function with $\chi|_{(-\infty,1]}\equiv1$, $\chi|_{[0,\infty)}\equiv 0$ and $\chi'|_{(-1,0)}<0$. Then, by definition $\beta$ is constant for $t\geq b$ and $\beta'(0)=-a$.
		
		First, we consider the angle $\theta$ at $t=0$ We have
		\[ \cos(\theta)=\frac{\overline{g}(\partial_s,\partial_t+\beta'(0)\partial_s)}{\lVert \partial_s\rVert\lVert\partial_t+\beta'(0)\partial_s\rVert}=\frac{\beta'(0)}{\sqrt{1+\beta'(0)^2}}=-\frac{a}{\sqrt{1+a^2}}\to 0 \]
		as $a\to 0$. Hence, for $a$ sufficiently small, we have $\theta<\frac{\pi}{2}+\varepsilon$.
		
		Next, note that the second fundamental form on the face $\{\beta(t)=s\}$ is given by the negatives of the expressions in Lemma \ref{L:II_graph_warped} with $\alpha=\beta^{-1}$ (and $R(s)=B(s)$) whenever $\beta'(\beta^{-1}(s))\neq 0$, that is, for $\alpha(s)\in[0,b]$, as we need to consider the (in $t$-direction) downward pointing normal vector in our case. Thus, we need to estimate the expressions
		\[ -f'(\alpha)f(\alpha)-2\frac{f'(\alpha)}{f(\alpha)}{\alpha'}^2+\alpha''\quad\text{ and }\quad -f'(\alpha)f(\alpha)+\frac{B'}{B}\alpha'. \]
		The second expression is positive for $a$ and $\lambda_2$ sufficiently small, since we have $f'(\alpha)f(\alpha)<\lambda_2(1+\lambda_2b)$ and $B'<0$, while $\alpha'<0$ and $\alpha'\to -\infty$ as $a\to 0$.
		
		For the first expression we first calculate
		\[ \alpha'(s)=\frac{1}{\beta'(\beta^{-1}(s))}\quad\text{ and }\quad \alpha''(s)=-\frac{\beta''(\beta^{-1}(s))}{\beta'(\beta^{-1}(s))^3}. \]
		
		Moreover, we have
		\begin{align*}
			\beta''(t)=\frac{a}{b}\chi(t-b)+a\left(\frac{t}{b}-1\right)\chi'(t-b),
		\end{align*}
		so we obtain for the first expression
		\begin{align*}
			-f'(\alpha)f(\alpha)-2\frac{f'(\alpha)}{f(\alpha)}{\alpha'}^2+\alpha''&=\frac{1}{\beta'(\alpha)^2}\left( -\beta'(\alpha)^2f'(\alpha)f(\alpha)-2\frac{f'(\alpha)}{f(\alpha)}-\frac{\beta''(\alpha)}{\beta'(\alpha)} \right)
		\end{align*}
		We have 
		\[ \beta'(\alpha)^2f'(\alpha)f(\alpha)\leq a^2\lambda_2(1+\lambda_2b),\quad \frac{f'(\alpha)}{f(\alpha)}<\frac{\lambda_2}{1+\alpha\lambda_1} \]
		and
		\[ \frac{-\beta''(\alpha)}{\beta'(\alpha)}=\frac{1}{b-\alpha}-\frac{\chi'(\alpha-b)}{\chi(\alpha-b)}\geq \frac{1}{b-\alpha}. \]
		Hence, we obtain
		\begin{align*}
			-\beta'(\alpha)^2f'(\alpha)f(\alpha)-2\frac{f'(\alpha)}{f(\alpha)}-\frac{\beta''(\alpha)}{\beta'(\alpha)}&>-a^2\lambda_2(1+\lambda_2b)-2\frac{\lambda_2}{1+\lambda_1\alpha}+\frac{1}{b-\alpha}.
		\end{align*}
		It now follows that this expression is positive for $b<\frac{1}{2\lambda_2}$ and $a$ sufficiently small.
		
		Finally, the induced metric along the face $\{s=\beta(t)\}$ we obtain
		\[ (1+\beta'(t)^2f(t)^2)dt^2+f(t)^2B(\beta(t))^2ds_{n-1}^2. \]
		This metric converges, smoothly to $dt^2+f(t)^2B(0)^2ds_{n-1}^2$ as $a\to 0$, which has positive curvature for $\lambda_2$ sufficiently small. Hence, for $a$ sufficiently small, (2) holds.
	\end{proof}
	
	Combining Propositions \ref{P:handle1} and \ref{P:handle2} by gluing the two pieces together gives the following corollary.
	
	\begin{corollary}\label{C:handle}
		In the setting of Proposition \ref{P:handle1} there exists $\lambda\in(0,1)$ and a Riemannian metric $h$ of positive Ricci curvature on $[0,1]\times M$ such that the following holds:
		\begin{enumerate}
			\item $h$ induces the metric $\cone{g}$ on $\{0\}\times M$ with principal curvatures all at least $-\lambda$,
			\item $h$ induces a metric of positive sectional curvature of the form $ds^2+A(s)^2ds_{n-1}^2$ on the face $\{1\}\times D^n$ with non-negative second fundamental form,
			\item $h$ is of the form $dt^2+f(t)^2g'$ near the face $\{1\}\times M\setminus{D^n}^\circ$, where $g'$ is a metric of positive Ricci curvature on $\{1\}\times M\setminus{D^n}^\circ$ with positive second fundamental form on the boundary $\{1\}\times S^{n-1}$, and all derivatives of $f$ vanish at $t=1$.
		\end{enumerate}
	\end{corollary}
	
	\begin{proof}
		We consider the metric $g_1$ on $[0,1]\times M$ constructed in Proposition \ref{P:handle1} and the metric $g_2$ on $[0,1]\times (M\setminus {D^n}^\circ)$ constructed in Proposition \ref{P:handle2}, where the Riemannian manifold $(M,g)$ in Proposition \ref{P:handle2} is $(M\setminus {D^n}^\circ, {g_1}|_{\{1\}\times M\setminus {D^n}^\circ})$. Using Proposition \ref{P:gluing_corner} we glue $([0,1]\times M,g_1)$ to $([0,1]\times (M\setminus {D^n}^\circ),g_2)$ along the faces $\{1\}\times (M\setminus {D^n}^\circ)$ resp.\ $\{0\}\times (M\setminus {D^n}^\circ)$. Then, for $\nu$ and $\varepsilon$ in Proposition \ref{P:handle2} sufficiently small, all required properties of Proposition \ref{P:gluing_corner} are satisfied. 
		
		Moreover, since both metrics $g_1$ and $g_2$ are given by $ds^2+A(s)^2ds_{n-1}^2$ on the face $[0,1]\times S^{n-1}$ with $A''<0$ near the gluing area, it follows from Proposition \ref{P:gluing_corner} that the glued manifold is of the same form.
	\end{proof}

	\section{Core metrics on glued spaces}\label{S:glued_spaces}
	
	In this section we prove the following result, from which the main results will follow.
	
	The general setting we consider is the following: Let $B^q$ be a closed manifold and let $\overline{E}\xrightarrow{\pi} B$ be a linear $D^{p}$-bundle over $B$. Denote by $E=\partial\overline{E}$ the total space of the corresponding linear $S^{p-1}$-bundle. Further, let $N$ be a manifold with boundary diffeomorphic to $E$.
	\begin{theorem}\label{T:NE_core_mtrc}
		Suppose $q\geq 2, p\geq 3$ and that $B$ admits a core metric. Further, suppose that $N$ admits a metric of positive Ricci curvature with convex boundary such that the induced metric on the boundary is of the form $g_r$ for $r<r_0$ as in Proposition \ref{P:Ric_bundles} (with $\check{g}$ the core metric on $B$ and $\hat{g}=ds_{p-1}^2$), thus having positive Ricci curvature. Then the manifold
		\[ N\cup_{E}\overline{E} \]
		admits a core metric.
	\end{theorem}
	Note that, by Proposition \ref{P:deform}, it suffices in fact to require for the induced metric on $\partial N$ that it lies in the same path component as a submersion metric in the space of Ricci-positive metrics on $E$.
	
	The proof of Theorem \ref{T:NE_core_mtrc} will not work in the case $p=2$, since the construction of the metric $\overline{g}_2$ of Subsection \ref{SS:IE} below, which \textquotedblleft transfers\textquotedblright\ the positivity of the second fundamental form from the fibre to the base, uses Proposition \ref{P:bdl_warping}. However, if we use Proposition \ref{P:bdl_warping_S1} instead, we obtain the following special case:
	
	\begin{theorem}\label{T:BxS2_core_mtrc}
		Suppose that $q\geq 2$ and $B^q$ admits a core metric. Then $B\times S^2$ admits a core metric.
	\end{theorem}
	
	In fact, Theorem \ref{T:BxS2_core_mtrc} follows from Theorem \ref{T:core_bdl}, which in turn follows from Theorem \ref{T:NE_core_mtrc}, if $q>2$. In the case $q=2$ we have $B=S^2$, so that the only new manifold obtained in Theorem \ref{T:BxS2_core_mtrc} is $S^2\times S^2$. We will nevertheless give the proof for general $B$ as it is essentially the same.
	
	\subsection{Topological decomposition}\label{SS:top_dec}
	
	Let $\overline{\varphi}\colon D^q\times D^p\hookrightarrow \overline{E}$ be a local trivialization covering the embedding $D^q\subseteq B$ of the round hemisphere into $B$. Let $\varphi\colon D^q\times S^{p-1}\hookrightarrow E$ be its restriction to $E$. Then we can decompose $\overline{E}$ into $\overline{E}\setminus\overline{\varphi}(D^q\times D^p)^\circ=\pi^{-1}(B\setminus{D^q}^\circ)$ and $D^q\times D^p$, where the gluing map is given by $\overline{\varphi}|_{S^{q-1}\times D^p}$. Hence, we obtain the following decomposition of the space $N\cup_{E}\overline{E}$:
	
	\[
	\begin{tikzcd}[cells={nodes={draw=black,anchor=center,minimum height=2em}}]
		D^q\times D^p \arrow[dash]{rr}{S^{q-1}\times D^p}\arrow[dash,swap]{dr}{D^q\times S^{p-1}}& & \overline{E}\setminus\overline{\varphi}(D^q\times D^p)^\circ\\
		& N\arrow[dash,swap]{ur}{E\setminus\varphi(D^q\times S^{p-1})^\circ}&
	\end{tikzcd}
	\]
	Here edges stand for a gluing along the indicated parts of the boundaries.
	
	In the following we will define explicit metrics on each part. To glue the parts together, we introduce two \textquotedblleft transition regions\textquotedblright\ as follows:
	
	\[
	\begin{tikzcd}[cells={nodes={draw=black,anchor=center,minimum height=2em}}]
		D^{p+q}\arrow{d}{S^{p+q-1}} & & \\
		I\times S^{p+q-1} \arrow{rr}{S^{q-1}\times D^p}\arrow[swap]{dr}{D^q\times S^{p-1}}& & \overline{E}\setminus\overline{\varphi}(D^q\times D^p)^\circ\arrow{dl}{E\setminus\varphi(D^q\times S^{p-1})^\circ}\\
		&I\times E\arrow{d}{E}&\\
		&N&
	\end{tikzcd}
	\]
	
	Here $I$ denotes a closed interval and the arrows indicate the gluing direction, that is, a segment \[\begin{tikzcd}[cells={nodes={draw=black,anchor=center,minimum height=2em}}]
		M_1\arrow{r}{N}&M_2
	\end{tikzcd}\]
	stands for gluing $M_2$ to $M_1$ along a common boundary component $N$. To preserve positive Ricci curvature along this gluing using Proposition \ref{P:gluing} we will require that the induced metrics on $N\subseteq M_1$ and $N\subseteq M_2$ are the same and the sum of second fundamental forms $\II_{N\subseteq M_1}+\II_{N\subseteq M_2}$ is non-negative. In addition, to obtain a convex boundary when removing the top $D^{p+q-1}$, we will make sure that the second fundamental form on each remaining boundary component of $M_1$ is non-negative.
	
	We will now describe the metric on each piece.
	
	\subsection{The metric on $N$}\label{SS:N}
	
	By the assumptions of Theorem \ref{T:NE_core_mtrc} there exists a metric $g_N$ on $N$ of positive Ricci curvature with convex boundary that is of the form $g_{r}^\theta$ on the boundary as defined in Subsection \ref{SS:fibre_bdls}. Here we assume that the base metric $\check{g}$ on $B$ is a core metric and the fibre metric is the round metric on $S^{p-1}$.
	
	In a first step, for given $\varepsilon_1,\varepsilon_2>0$ we deform the metric $\check{g}$ on $B$ through metrics of positive Ricci curvature to the metric $\cone{g}$ constructed in Lemma \ref{L:cone_metric}. By the expression \eqref{EQ:submersion_metric} we obtain a corresponding deformation of metrics on $E$. Further, it follows from the formulas \eqref{EQ:Ric_bundles} for the Ricci curvatures that for $r$ sufficiently small this deformation is through metrics of positive Ricci curvature.
	
	We can modify the connection $\theta$ used in the construction of the metric $g_r^\theta$ to be the connection induced by the product structure over the hemisphere $D^q\subseteq B$, e.g.\ by taking convex combinations. Again by the formulas \eqref{EQ:Ric_bundles} it follows that for $r$ sufficiently small, this deformation is through metrics of positive Ricci curvature. By observing that decreasing $r$ also preserves the positivity of the Ricci curvature, it follows that we can smoothly deform the metric $g_r^\theta$ which corresponds to the metric $\check{g}$ on $B$, to a metric $g_{r'}^{\theta'}$ for all $r'$ sufficiently small, which corresponds to the metric $\cone{g}$ on $B$ and such that this metric is a product metric on $\varphi(D^q\times S^{p-1})$. By Proposition \ref{P:deform} and and rescaling, it follows that we can attach a cylinder $[0,1]\times E$ to $N$ so that the new boundary is still convex and induces the metric $g_{r'}^{\theta'}$ on the boundary. We denote $r'$ and $\theta'$ again by $r$ and $\theta$, respectively.
	
	Hence, to summarize, for any $r>0$ sufficiently small there exists a metric $\overline{g}_1$ of positive Ricci curvature on $N$ such that the boundary is convex and the induced metric on the boundary is given by the submersion metric $g_{r}^{\theta}$, where the metric considered on $B$ is $\cone{g}$, the metric on $S^{p-1}$ is $ds_{p-1}^2$ and the metric $g_{r}^{\theta}$ is a product metric over $D^q\subseteq B$.

	\subsection{The metric on $I\times E$}\label{SS:IE}
	
	In this subsection, we construct a metric on the cylinder $I\times E$ using the main building block of Corollary \ref{C:handle}. This metric will transition between $\overline{g}_1$ and a metric to which we can glue $\overline{E}\setminus\overline{\varphi}(D^q\times D^p)^\circ$ such that the resulting boundary, which is given by $D^q\times S^{p-1}\cup_{S^{q-1}\times S^{p-1}}S^{q-1}\times D^p\cong S^{p+q-1}$ has positive Ricci curvature and is convex.
	
	We start with the metric $h$ on $[0,1]\times B$ constructed in Corollary \ref{C:handle}. For $r>0$ and the connection $\theta$ of Subsection \ref{SS:N} we consider the metric $h_r^\theta$ on the pull-back of $E\to B$ to $[0,1]\times B$ along the projection $[0,1]\times B\to B$, i.e.\ to the bundle
	\[[0,1]\times E\xrightarrow{\mathrm{id}\times \pi}[0,1]\times B,\]
	where we equip the base $[0,1]\times B$ with the metric $h$ and we consider the pull-back connection of $\theta$. In particular, by the assumption on the connection $\theta$, this metric is a product metric over $[0,1]\times D^q$. For $r>0$ sufficiently small this metric has positive Ricci curvature by Proposition \ref{P:Ric_bundles} and induces the metric $g_{r}^\theta$ on $\{0\}\times E$.
	
	By Lemma \ref{L:II_bdl}, the second fundamental form of this metric is the pull-back of the second fundamental form on the base. Hence, for $r$ sufficiently small, we can glue this piece along $\{0\}\times E$ to the metric on $[0,1]\times E$ constructed in Proposition \ref{P:bdl_warping} along $\{1\}\times E$ (after rescaling by $R^{-1}$) using Proposition \ref{P:gluing}. We denote this metric by $\overline{g}_2$. 
	
	To summarize, for any $\nu>0$ and $r',r>0$ sufficiently small there exists a metric $\overline{g}_2$ of positive Ricci curvature on $[0,1]\times E$ with the following properties:
	\begin{enumerate}
		\item The boundary component $\{0\}\times E$ is, after rescaling, isometric to $g_{r'}^\theta$ and has principal curvatures all at least $-\nu$,
		\item The induced metric on the face $\{1\}\times \pi^{-1}(D^q)\cong\{1\}\times D^q\times S^{p-1}$ is a product metric of the form $ds^2+A(s)^2ds_{q-1}^2+r^2ds_{p-1}^2$ of non-negative Ricci curvature with non-negative second fundamental form,
		\item The metric is of the form $dt^2+f(t)^2\pi^*g'+r^2 \mathrm{pr}_{\mathcal{V}}^*ds_{p-1}^2$ near the face $\{1\}\times \pi^{-1}(B\setminus{D^q}^\circ)$, where $g'$ is a metric of positive Ricci curvature on $\{1\}\times B\setminus{D^q}^\circ$ with positive second fundamental form on the boundary $\{1\}\times S^{q-1}$, and all derivatives of $f$ vanish at $t=1$.
	\end{enumerate}
	
	\subsection{The metric on $\overline{E}\setminus\overline{\varphi}(D^q\times D^p)^\circ$}\label{SS:E-DqDp}
	
	We view the restriction $\overline{E}\setminus\overline{\varphi}(D^q\times D^p)^\circ\xrightarrow{\pi} B\setminus{D^q}^\circ$ as a linear $D^p$-bundle. For $r,R>0$ we consider the metric $g_r'$ on this bundle, where the metric on the base is the metric $R^2g'$ in the preceding subsection, (the restriction of) the connection $\theta$, and the fibre metric on $D^p$ is a warped product metric $\hat{g}=dt^2+h(t)^2ds_{p-1}^2$ for some function $h\colon[0,t_0]\to[0,\infty)$ with $h(0)=0$, $h'(0)=1$, $h$ is odd at $t=0$, $h''<0$, $h(t_0)=1$ and all derivatives of $h$ vanish at $t=t_0$. Then the metric $\hat{g}$ has positive sectional curvature.
	
	By construction the metric $g_r'$ is a product on the face $\pi^{-1}(S^{q-1})$ of the form
	\[{R'}^2ds_{q-1}^2+r^2\hat{g}={R'}^2ds_{q-1}^2+r^2(dt^2+h(t)^2ds_{p-1}^2)=dt^2+{R'}^2ds_{q-1}^2+r^2h(\tfrac{t}{r})^2ds_{p-1}^2,\]
	and it is globally of the form
	\[ g_r'=R^2\pi^*g'+r^2\mathrm{pr}_{\mathcal{V}}^*(dt^2+h(t)^2ds_{p-1}^2)=dt^2+R^2\pi^*g'+r^2h(\tfrac{t}{r})^2\mathrm{pr}_{\mathcal{V}}^*ds_{p-1}^2. \]
	Moreover, by Lemma \ref{L:II_bdl}, the second fundamental form along the face $\pi^{-1}(S^{q-1})$ is non-negative. We set $\overline{g}_3=g_r'$.
	
	To summarize, for any $R>0$ and any $r>0$ sufficiently small there exists a metric $\overline{g}_3$ of positive Ricci curvature on $\overline{E}\setminus\overline{\varphi}(D^q\times D^p)^\circ$ with the following properties:
	\begin{enumerate}
		\item $\overline{g}_3$ is of the form $dt^2+R^2\pi^*g'+r^2h(\tfrac{t}{r})^2\mathrm{pr}_{\mathcal{V}}^*ds_{p-1}^2$, where $h\colon[0,t_0]\to[0,\infty)$ is a smooth function with $h''<0$, $h(t_0)=1$ and vanishing derivatives at $t=t_0$,
		\item The metric on the face $\pi^{-1}(S^{q-1})$ is a doubly warped product metric of the form $dt^2+{R'}^2ds_{q-1}^2+r^2h(\frac{t}{r})^2ds_{p-1}$ with non-negative second fundamental form.
	\end{enumerate}
	
	\subsection{The metric on $I\times S^{p+q-1}$}\label{SS:ISpq}
	
	To transition to the round metric in the last step, we assume that we are given doubly warped product metric on $S^{p+q-1}$, i.e.\ the metric is of the form
	\[ g_{A,B}=ds^2+A(s)^2ds_{q-1}^2+B(s)^2ds_p^2 \]
	for $s\in[0,s_0]$, where $A,B\colon[0,s_0]\to[0,\infty)$ are smooth functions satisfying the smoothness conditions \ref{EQ:dw_boundary_0} and \ref{EQ:dw_boundary_t0} (with $A$ and $B$ replaced by $f$ and $h$, respectively).
	
	We also assume that $A'',B''<0$ on $(0,s_0]$ and $[0,s_0)$, respectively, and $A'''(0),B'''(s_0)<0$, which is equivalent to $g_{A,B}$ having positive sectional curvature by Lemma \ref{L:sphere_dbl_warped}. Then, by Lemma \ref{L:sphere_dbl_warped} in combination with Proposition \ref{P:deform}, we obtain for any $\nu>0$ a metric $\overline{g}_4$ of positive Ricci curvature on $I\times S^{p+q-1}$ satisfying 
	\begin{enumerate}
		\item $\overline{g}_4$ induces the metric $g_{A,B}$ on $\{0\}\times S^{p+q-1}$ and the principal curvatures are all at least $-\nu$,
		\item $\overline{g}_4$ induces a multiple of the round metric on $\{1\}\times S^{p+q-1}$ and the boundary is convex.
	\end{enumerate}
	
	\subsection{Proof of Theorems \ref{T:NE_core_mtrc} and \ref{T:BxS2_core_mtrc}}
	
	By using the decomposition of Subsection \ref{SS:top_dec} and the metrics constructed in Subsections \ref{SS:N}--\ref{SS:ISpq} we can now prove Theorem \ref{T:NE_core_mtrc}.
	
	\begin{proof}[Proof of Theorem \ref{T:NE_core_mtrc}]		
		We start with the metric $\overline{g}_1$ on $N$ constructed in Subsection \ref{SS:N}, i.e.\ the boundary is convex and the metric on the boundary is given by $g_{r}^\theta$ for $r>0$ sufficiently small and the connection is chosen so that the metric is a product over $D^q\subseteq B$.
		
		Using Proposition \ref{P:gluing}, we glue this metric to the metric $\overline{g}_2$ constructed in Subsection \ref{SS:IE} by gluing $E=\partial N$ to $\{0\}\times E\subseteq [0,1]\times E$, where we choose the parameter $\nu$ for the metric $\overline{g}_2$ sufficiently small so that the hypothesis of Proposition \ref{P:gluing} are satisfied.
		
		Now given this metric, we glue it to the metric $\overline{g}_3$ on $\overline{E}\setminus\overline{\varphi}(D^q\times D^p)^\circ$ constructed in Subsection \ref{SS:E-DqDp} by gluing $\{1\}\times E\setminus\varphi(D^q\times S^{p-1})^\circ\subseteq [0,1]\times E$ to $E\setminus\varphi(D^q\times S^{p-1})^\circ\subseteq \overline{E}\setminus\overline{\varphi}(D^q\times D^p)^\circ$. Here we choose the parameter $R$ for the metric $\overline{g}_3$ such that these two metrics coincide on the gluing area. Since the warping functions $f$ and $h$ for the metrics $\overline{g}_2$ and $\overline{g}_3$, respectively, have vanishing derivatives at the boundary, the glued metric is already smooth, so no further smoothing is required.
		
		Hence, we have constructed a metric of positive Ricci curvature on the manifold 
		\[N\cup_{E\setminus\varphi(D^q\times S^{p-1})^\circ}\overline{E}\setminus\overline{\varphi}(D^q\times D^p)^\circ\]
		which has non-negative second fundamental form on the boundary and which is of the form
		\[ ds^2+A(s)^2ds_{q-1}^2+B(s)^2ds_{p-1}^2 \]
		for $s\in[0,s_0]$ and some $s_0>0$ on the boundary $S^{p+q-1}$ with $A'',B''\leq 0$ and $A''<0$ resp.\ $B''<0$ in a neighbourhood of $s=0$ resp.\ $s=s_0$ with $A'''(0),B'''(s_0)<0$. We slightly modify this metric by first applying Proposition \ref{P:deform_II>=0} to obtain a strictly positive second fundamental form on the boundary and by subsequently perturbing the metric slightly preserving positive Ricci curvature and positive second fundamental form such that additionally we have $A'',B''<0$ on $(0,s_0]$ resp.\ $[0,s_0)$, thus obtaining positive sectional curvature on the boundary by Lemma \ref{L:doubly_warped_curv}.
		
		Finally, we glue this metric along the boundary $S^{p+q-1}$ to the metric $\overline{g}_4$ constructed in Subsection \ref{SS:ISpq} along $\{0\}\times S^{p+q-1}$. Here we choose the parameter $\nu$ for the metric $\overline{g}_4$ sufficiently small so that we can apply Proposition \ref{P:gluing} for this gluing. Thus, we obtain a metric of positive Ricci curvature on $(N\cup_{E}\overline{E})\setminus {D^{p+q}}^\circ$ with round and convex boundary. By Lemma \ref{L:core_equ} this implies that $N\cup_{E}\overline{E}$ admits a core metric.
	\end{proof}
	\begin{proof}[Proof of Theorem \ref{T:BxS2_core_mtrc}]
		For the proof of Theorem \ref{T:BxS2_core_mtrc} we also use the decomposition given in Subsection \ref{SS:top_dec}, with the only difference that there is no manifold $N$ and we directly start with a metric on $I\times E=I\times B\times S^1$.
		
		On this part we construct a metric along the same lines as in Subsection \ref{SS:IE} (where the connection is trivial in our case), where we use Proposition \ref{P:bdl_warping_S1} instead of Proposition \ref{P:bdl_warping} applied to the metric $\cone{g}$ on $B$, so that we obtain a metric on $B\times D^2$. We then glue this metric to $[0,1]\times B\times S^1$ equipped with a multiple of the metric $h+r^2ds_1^2$, where $h$ is the metric constructed in Corollary \ref{C:handle} by using Proposition \ref{P:gluing_Ric>=0}. For that we will not perform the final deformation in the proof of Proposition \ref{P:gluing_Ric>=0} so that the metric on $[0,1]\times B\times S^1$ is unchanged outside a neighbourhood of $\{0\}\times B\times S^1$, so that we merely have non-negative Ricci curvature and some points with positive Ricci curvature.
		
		The rest of the proof is entirely similar to the proof of Theorem \ref{T:NE_core_mtrc}. Note that the metric constructed in this way has strictly positive Ricci curvature after applying the deformation of Proposition \ref{P:deform_II>=0_bdry} as in the proof of Theorem \ref{T:NE_core_mtrc}.
	\end{proof}
	
	\subsection{Proof of Theorem \ref{T:core_bdl}}
	
	Finally, we can apply Theorem \ref{T:NE_core_mtrc} to prove Theorem \ref{T:core_bdl}.
	\begin{proof}[Proof of Theorem \ref{T:core_bdl}]
		Let $E\xrightarrow{\pi} B^q$ be a fibre bundle with fibre $F^p$ and structure group $G$ satisfying the hypotheses of Theorem \ref{T:core_bdl}. For $r>0$ we equip $E$ with the metric $g_r^\theta$ with respect to the metrics $\check{g}$ and $\hat{g}$ on $B$ and $F$, respectively, and we choose $\theta$ so that it is induced from the product structure on $\pi^{-1}(D^q)\cong D^q\times F$, so that $g_r^\theta$ is given by $ds_q^2|_{D^q}+r^2\hat{g}$ on this part (where $D^q$ is identified with a hemisphere in $S^q$).
		
		Hence, the metric on $E\setminus \pi^{-1}(D^q)^\circ$ induces the metric $ds_{q-1}^2+r^2\hat{g}$ on the boundary and the boundary has non-negative second fundamental form by Lemma \ref{L:II_bdl}. By Proposition \ref{P:deform_II>=0} we can deform the metric so that the second fundamental form is strictly positive. Then, by Proposition \ref{P:deform} in combination with Proposition \ref{P:gluing}, we can attach a cylinder to the boundary, so that the metric on the new boundary has again positive second fundamental form and the induced metric is, after rescaling, given by $ds_{q-1}^2+\hat{g}'$.
		
		Now we are in the setting of Theorem \ref{T:NE_core_mtrc}, where we set $N$ in Theorem \ref{T:NE_core_mtrc} as $E\setminus \pi^{-1}(D^q)^\circ$ and $\overline{E}$ in Theorem \ref{T:NE_core_mtrc} as the trivial bundle $F\times D^q$ (and note that $q$ in Theorem \ref{T:NE_core_mtrc} is given by $p$ and vice versa). Hence, we obtain a core metric on the glued space
		\[E\setminus \pi^{-1}(D^q)^\circ\cup_{S^{q-1}\times F}(D^q\times F)\cong E.\]
	\end{proof}

	\section{Applications}\label{S:app}
	
	We will discuss three applications of Theorem \ref{T:NE_core_mtrc}: Linear sphere bundles (Subsection \ref{SS:sph_bdls}), projective bundles (Subsection \ref{SS:Proj_bdl}) and the Wu manifold (Subsection \ref{SS:Wu}). We will then show that we can represent various different bordism classes by manifolds of positive Ricci curvature in Subsection \ref{SS:bordism}.

	\subsection{Linear sphere bundles}\label{SS:sph_bdls}
	
	In this subsection we show the following result:
	\begin{theorem}\label{T:sph_bdl}
		Let $E\xrightarrow{\pi} B^q$ be a linear sphere bundle with fibre $S^p$, $p\geq 2$. If $B$ is closed and admits a core metric, then $E$ admits a core metric.
	\end{theorem}
	
	Theorem \ref{T:sph_bdl} follows from Theorem \ref{T:core_bdl} when $q\geq 3$ by setting $\hat{g}=\hat{g}'=ds_{p}^2$. For the case $q=2$ we establish the following result:
	
	\begin{proposition}\label{P:sph_bdl_section}
		Let $E\xrightarrow{\pi} B^q$ be a linear $S^p$-bundle with $p\geq 3$ such that $B$ is closed and admits a core metric. Suppose $\pi$ admits a section, or, equivalently, the structure group reduces to $O(p)$. Then $E$ admits a core metric.
	\end{proposition}
	\begin{proof}
		Since $\pi$ admits a section, we can decompose it into the union of two linear disc bundles. Let $\underline{E}\to B$ denote one of these disc bundles, i.e.\ we obtain $E$ by gluing $\underline{E}\cup_{\partial\underline{E}} (-\underline{E})$. We equip $\underline{E}$ with a submersion metric $g_r$ with totally geodesic fibres as in Subsection \ref{SS:fibre_bdls}. Here we set $\check{g}$ the core metric on $B$, $\hat{g}$ the induced metric of a round hemisphere on $D^p\subseteq S^p$, and an arbitrary principal connection $\theta$.
		
		By Proposition \ref{P:Ric_bundles}, the metric $g_r$ has positive Ricci curvature for $r$ sufficiently small, and since $g_r$ defines a smooth metric on $\underline{E}\cup_{\partial\underline{E}}(-\underline{E})=E$, the boundary $\partial \underline{E}$ is totally geodesic. Hence, it follows from Theorem \ref{T:NE_core_mtrc} that $E$ admits a core metric, by setting $N$ and $\overline{E}$ in Theorem \ref{T:NE_core_mtrc} both as $\underline{E}$.
	\end{proof}
	
	\begin{proof}[Proof of Theorem \ref{T:sph_bdl}]
		The case $q\geq 3$ directly follows from Theorem \ref{T:core_bdl}. Now suppose $q=2$. Since $B$ is a closed $2$-manifold that admits a core metric, we have $B=S^2$. Thus, isomorphism classes of linear $S^p$-bundles over $B$ are in bijection with elements of $\pi_2(\mathrm{B}O(p+1))\cong\pi_1(O(p+1))$. Since the map $\pi_1(O(p))\to\pi_1(O(p+1))$ induced by the inclusion is an isomorphism for $p\geq 3$, we can apply Proposition \ref{P:sph_bdl_section} to obtain a core metric on $E$, provided $p\geq 3$.
		
		Finally, suppose $p=q=2$. Since $\pi_1(O(3))\cong\Z/2$, there are precisely two isomorphism classes of linear $S^2$-bundles over $S^2$. For the trivial one we obtain a core metric from Theorem \ref{T:BxS2_core_mtrc}. If $\pi$ is the non-trivial bundle, we have a diffeomorphism $E\cong\C P^2\#(-\C P^2)$. By \ref{EQ:core1} and \ref{EQ:core3} this manifold also admits a core metric.
	\end{proof}

	\subsection{Projective bundles}\label{SS:Proj_bdl}
	
	As another application we show that Theorem \ref{T:core_bdl} can be applied to projective bundles. For that, let $\K\in\{\C,\Quat,\mathbb{O}\}$ and consider a projective space $\K P^n$, where $n=2$ if $\K=\mathbb{O}$. These are given as the homogeneous spaces
	\begin{align*}
		\C P^n&\cong \bigslant{U(n+1)}{U(n)\times U(1)},\\
		\Quat P^n&\cong \bigslant{Sp(n+1)}{Sp(n)\times Sp(1)},\\
		\mathbb{O} P^2&\cong\bigslant{F_4}{Spin(9)},
	\end{align*}
	see e.g.\ \cite[Chapter 3]{Be78}. We say that a fibre bundle $E\to B$ with fibre $\K P^n$ is \emph{projective}, if it is homogeneous with respect to these identifications, i.e.\ if the structure group is contained in $U(n+1)$ (resp.\ $Sp(n+1)$, $F_4$) when $\K=\C$ (resp.\ $\K=\Quat$, $\mathbb{O}$).
	
	\begin{theorem}\label{T:proj_bdl}
		Let $E\to B^q$ be a projective bundle with fibre $\K P^n$ and suppose that $B$ is closed and admits a core metric. Then $E$ admits a core metric.
	\end{theorem}
	
	To prove Theorem \ref{T:proj_bdl}, we recall two important metrics on $\K P^n$. First, we have the metric $g_0$ induced from a biinvariant metric on $G_0=U(n+1)$ for $\K=\C$ (resp.\ $G_0=Sp(n+1)$, $F_4$ for $\K=\Quat$, $\mathbb{O}$). By construction, this metric is invariant under the action of $G_0$, and it has positive Ricci curvature (in fact, it is Einstein and has positive sectional curvature), see e.g.\ \cite[Section 7.D]{Be87}.
	
	Second, we have a core metric $g_1$ on $\K P^n$ constructed in \cite{Bu19} using doubly warped submersion metrics. Based on earlier work by Cheeger \cite{Ch73}, an alternative construction using cohomogeneity-one actions, which results in essentially the same metric as in \cite{Bu19}, was given in \cite{RW23}. We will take the latter approach as it will allow us to express the metric $g_0$ more easily in this context. We refer to \cite[Section 6.3]{AB15} for general background on metrics on cohomogeneity-one manifolds.
	
	To obtain an action of cohomogeneity one on $\K P^n$ we restrict the action to $G=U(n)$ for $\K=\C$ (resp.\ $G=Sp(n)$, $Spin(9)$ for $\K=\Quat$, $\mathbb{O}$). Then the quotient $\K P^n/G$ is an interval $I=[-1,1]$ with principal isotropy $H$ over $(-1,1)$ and non-principal isotropy $K_{\pm}$ over $\pm1$ given as in Table \ref{table}, see e.g.\ \cite[Table 1]{BM17}. We note that all inclusions $H\subseteq K_{\pm}\subseteq G$ are the standard inclusion except for $Spin(7)\hookrightarrow Spin(8)$ (see e.g.\ \cite[9.84]{Be87}).
	
	\renewcommand{\arraystretch}{1.5}
	\begin{table}[h!]
		\centering
		\begin{tabular}{|l|l|l|l|l|}
			\hline
			& $G$ & $K_-$ & $K_+$ & $H$\\\hline\hline
			$\C P^n$ & $U(n)$ & $U(n)$ & $U(n-1)U(1)$ & $U(n-1)$\\
			$\Quat P^n$ & $Sp(n)$ & $Sp(n)$ & $Sp(n-1)Sp(1)$ & $Sp(n-1)$\\
			$\mathbb{O}P^2$ & $Spin(9)$ & $Spin(9)$ & $Spin(8)$ & $Spin(7)$\\\hline
		\end{tabular}
		\medskip
		\caption{Cohomogeneity one structure of projective spaces.}
		\label{table}
	\end{table}
	
	To construct a $G$-invariant metric on $\K P^n$ let $L$ be the left-$G$-invariant and right-$K$-invariant metric on $G$ so that the induced metric on $G/H\cong S^{nd-1}$ is the round metric, where $d=\dim_\R\K\in\{2,4,8\}$. Let $\mathfrak{g}=\mathfrak{k}_+\oplus\mathfrak{m}$ and $\mathfrak{k}_+=\mathfrak{h}\oplus\mathfrak{p}$ be $L$-orthogonal decompositions. For smooth functions $f,h\colon I\to[0,\infty)$ we define the metric
	\[ g_{f,h}=dt^2+f(t)^2 L|_{\mathfrak{p}}+h(t)^2 L|_{\mathfrak{m}} \]
	on $I\times G/H$. This metric induces a smooth metric on $\K P^n$, which we will again denote by $g_{f,h}$ if and only if
	\begin{enumerate}
		\item $f$ and $h$ are odd functions at $t=-1$ with $f'(-1)=h'(-1)=1$,
		\item $f$ is an odd function at $t=1$ with $f'(1)=-1$ and $h$ is an even function at $t=1$ with $h(1)>0$,
		\item $f$ and $h$ are strictly positive on $(-1,1)$.
	\end{enumerate}
	This follows for example from \cite[Theorem 1]{Ve99} and \cite[Proposition 2.6]{Bu19}. Moreover, the Ricci curvatures of $g_{f,h}$ are given as follows, see e.g.\ \cite[Lemma 2.5]{Bu19}:
	\begin{align*}
		\Ric(\partial_t,\partial_t)&=-(d-1)\frac{f''}{f}-(n-1)d\frac{h''}{h},\\
		\Ric(\tfrac{V}{f},\tfrac{V}{f})&=-\frac{f''}{f}+(d-2)\frac{1-{f'}^2}{f^2}-(n-1)d\frac{f'h'}{fh}+(n-1)d\frac{f^2}{h^4},\\
		\Ric(\tfrac{X}{h},\tfrac{X}{h})&=-\frac{h''}{h}+((n-1)d-1)\frac{1-{h'}^2}{h^2}-(d-1)\frac{f'h'}{fh}+3(d-1)\frac{1}{h^2}-2(d-1)\frac{f^2}{h^4},\\
		\Ric(\partial_t,X)&=\Ric(\partial_t,V)=\Ric(X,V)=0,
	\end{align*}
	where $V\in \mathfrak{p}$ and $X\in\mathfrak{m}$.
	
	We can now describe the metrics $g_0$ and $g_1$ in this context. Indeed, $g_0$ is given, up to a scalar multiple, as $g_{f_0,h_0}$ for
	\begin{align*}
		f_0(t)&=\frac{4}{\pi}\sin\left(\frac{\pi}{4}(t+1)\right)\cos\left(\frac{\pi}{4}(t+1)\right)=\frac{2}{\pi}\cos\left( \frac{\pi}{2}t \right),\\
		h_0(t)&=\frac{4}{\pi}\sin\left(\frac{\pi}{4}(t+1)\right),
	\end{align*}
	see e.g.\ \cite[Example 6.52]{AB15}. For the metric $g_1$ we set $f_1(t)=\frac{2}{\pi}\cos(t\tfrac{\pi}{2})=f_0(t)$ and we define $h_1$ as a smoothing on $[0,\varepsilon)$ as in \cite[Lemma 3.1]{RW23} of the function 
	\[ \tilde{h}_1(t)=\begin{cases}
		\frac{2}{\pi}\cos(\tfrac{\pi}{2}t),\quad &t\in[-1,0],\\
		\frac{2}{\pi},\quad &t\in[0,1].
	\end{cases} \]
	The smoothing is in such a way that it is $C^1$-close to $\tilde{h}_1$ and the second derivatives of $\tilde{h}_1$ around $0$ lie, up to an arbitrarily small error, in the interval $[\tilde{h}_{1-}''(0),\tilde{h}_{1+}''(0)]$. Since the Ricci curvatures are linear in the second derivatives of $f$ and $h$, we can therefore use the function $\tilde{h}_1$ instead of $h_1$ to calculate them. Indeed, a calculation shows that $g_{f_1,h_1}$ has positive Ricci curvature. Moreover, we have that the metric over $[-1,0]$ is given by
	\[ dt^2+\frac{2}{\pi}\cos^2\left(\frac{\pi}{2}t\right)^2L|_{\mathfrak{m}\oplus\mathfrak{p}}. \]
	Since $L|_{\mathfrak{m}\oplus\mathfrak{p}}$ defines the round metric on $G/H\cong S^{nd-1}$, we obtain, after rescaling, an isometric embedding of a round hemisphere, showing that $g_{f_1,h_1}$ is a core metric.
	
	As in the proof of Theorem \ref{T:sph_bdl}, the proof of Theorem \ref{T:proj_bdl} splits into the two cases $q\geq 3$ and $q=2$. For the second case we prove the following analogous result to Proposition \ref{P:sph_bdl_section}.
	
	\begin{proposition}\label{P:proj_bdl_section}
		Let $E\xrightarrow{\pi}B^q$ be a projective bundle with fibre $\K P^n$ such that $B$ is closed and admits a core metric. Suppose that the structure group of $\pi$ reduces to $G=U(n)$ (resp.\ $G=Sp(n)$, $Spin(9)$). Then $E$ admits a core metric.
	\end{proposition}
	\begin{proof}
		If $n=1$ and $\K=\C$, we have $\K P^n=S^2$, so $\pi$ is in fact a linear sphere bundle and we can apply Theorem \ref{T:sph_bdl}. Hence, we can assume that $n>1$ when $\K=\C$, so that the fibre dimension is at least $4$.
		
		Since the core metric $g_1$ is invariant under the structure group, we can use the construction of Subsection \ref{SS:fibre_bdls} to construct a submersion metric $g_r$ on $E$ with $\check{g}$ the core metric on $B$, $\hat{g}$ the core metric $g_1$ and arbitrary principal connection $\theta$. By Proposition \ref{P:Ric_bundles}, the metric $g_r$ has positive Ricci curvature for $r$ sufficiently small.
		
		Since the two parts $g_1|_{[-1,0]}$ on $D^{dn}$ and $g_1|_{[0,1]}$ on $\K P^n\setminus {D^{dn}}^\circ$ are preserved under the action of $G$, we can split the bundle $\pi$ into a linear $D^{dn}$-bundle $E_1$ and a fibre bundle $E_2$ with fibre $\K P^n\setminus{D^{dn}}^\circ$ and structure group $G$. As in the proof of Proposition \ref{P:sph_bdl_section}, the boundary $\partial E_2$, which is a linear sphere bundle over $B$, is totally geodesic in both $E_1$ and $E_2$. Hence, we can apply Theorem \ref{T:NE_core_mtrc} to obtain a core metric on $E$, where we set $N$ and $\overline{E}$ in Theorem \ref{T:NE_core_mtrc} as $E_2$ and $E_1$, respectively.
	\end{proof}
	
	\begin{proof}[Proof of Theorem \ref{T:proj_bdl}]
		First assume $q=2$, so $B=S^2$. Then isomorphism classes of projective bundles over $B$ are in bijection with elements of $\pi_2(\mathrm{B}G_0)\cong\pi_1(G_0)$. If $G_0=F_4$ or $Sp(n+1)$, this group is trivial, showing that $E=S^2\times\K P^n$, which admits a core metric by Theorem \ref{T:sph_bdl}. If $G_0=U(n+1)$, the map $\pi_1(U(n))\to\pi_1(U(n+1))$ induced by the inclusion is an isomorphism, so we can reduce the structure group to $U(n)$. Hence, by Proposition \ref{P:proj_bdl_section}, the manifold $E$ admits a core metric.
		
		Now assume $q\geq 3$. We will now show that the metrics $g_0$ and $g_1$ lie in the same path component of the space of Ricci-positive metrics on $\K P^n$, which implies by Theorem \ref{T:core_bdl} that $E$ admits a core metric.
		
		For $s\in[0,1]$ we define the function $h_s$ as a smoothing on $[\tfrac{1-s}{1+s},\tfrac{1-s}{1+s}+\varepsilon)$ of the function
		\[ \tilde{h}_s(t)= \begin{cases}
			\frac{4}{(s+1)\pi}\sin\left( \frac{\pi}{4}(t+1)(s+1) \right),\quad & t\in[-1,\tfrac{1-s}{1+s}],\\
			\frac{4}{(s+1)\pi},\quad & t\in[\tfrac{1-s}{1+s},1].
		\end{cases} \]
		As for the function $h_1$ we can use the function $\tilde{h}_s$ in the computations below, and we can choose the smoothing so that $h_s$ varies smoothly in $s$ between $h_0$ and $h_1$. We now claim that the metric $g_{f_0,h_s}=g_{f_1,h_s}$ has positive Ricci curvature for all $s$. For that, we first note that an explicit calculation shows that the Ricci curvatures are positive at $t=\pm1$. Further, on $(-1,1)$ we have $h_s''\leq 0$, $h_s'\in[0,1)$ and $f_0''<0$, $f_0'\in[0,1]$, $f_0\in[0,1]$. Hence, the Ricci curvature is positive if we can show that
		\[ \frac{f_0^2}{h_s^4}-\frac{f_0'h_s'}{f_0 h_s}\geq 0. \]
		For $t\in[\tfrac{1-s}{1+t},1]$ this expression is positive. For $t\in[-1,\frac{1-s}{1+s}]$ we set $t_s=\frac{\pi}{4}(t+1)(s+1)$ and obtain
		\begin{align*}
			\frac{f_0(t)^2}{h_s(t)^4}-\frac{f_0'(t)h_s'(t)}{f_0(t) h_s(t)}=\frac{(s+1)\pi^2}{16\sin(t_s)}\left( \cos^2(t_0)\left( (s+1)^3\frac{\sin(t_0)^2}{\sin(t_s)^3}-\frac{\cos(t_s)}{\sin(t_0)\cos(t_0)} \right)+\frac{\sin(t_0)\cos(t_s)}{\cos(t_0)} \right).
		\end{align*}
		By the concavity of the sine function and the monotonicity of the cosine function, we have $\sin(t_s)\leq (1+s)\sin(t_0)$ and $\cos(t_s)\leq \cos(t_0)$. It follows that the innermost bracket is non-negative, so that also the overall expression is non-negative.
	\end{proof}

	\subsection{The Wu manifold}\label{SS:Wu}
	
	In this subsection we consider the Wu manifold $W^5$ \cite{Wu50}, which is the closed, simply-connected $5$-manifold defined as the homogeneous space $W=SU(3)/SO(3)$. It is among one of the \textquotedblleft elementary\textquotedblright\ manifolds in Barden's classification of closed, simply-connected 5-manifolds \cite{Ba65}. The invariants of $W$ are given by $H_2(W)\cong\Z/2$ and $w_2(W)\neq0$ by \cite[Lemma 1.1]{Ba65}. The goal of this section is to prove Theorem \ref{T:Wu}
	
	We will use the following alternative description of $W$: Let $S^2\ttimes D^3$ and $S^2\ttimes S^2$ denote the total space of the unique non-trivial linear $D^3$- and $S^2$-bundle over $S^2$, respectively. Then we have $\partial (S^2\ttimes D^3)=S^2\ttimes S^2$. Further, we have a diffeomorphism $S^2\ttimes S^2\cong \C P^2\#(-\C P^2)$ as we can identify both spaces with the space obtained from $[-1,1]\times S^3$ by collapsing the spheres $\{\pm1\}\times S^3$ along the Hopf fibration $H\colon S^3\to S^2$. Then the pieces $[-1,0]\times S^3$ and $[0,1]\times S^3$ define $\C P^2\setminus{D^4}^\circ$ and $(-\C P^2\setminus{D^4}^\circ)$, respectively, while the projection $H\circ\mathrm{pr}_{S^3}\colon I\times S^3\to S^2$ defines a linear $S^2$-bundle over $S^2$ (cf.\ \cite[Lemma 1]{Wa64}).
	
	Now let $\phi\colon (\C P^2\#(-\C P^2))\to (\C P^2\#(-\C P^2))$ be a diffeomorphism that induced the map $(x,y)\mapsto (x,-y)$ according to the splitting $H_2(\C P^2\#(-\C P^2))\cong H_2(\C P^2)\oplus H_2(\C P^2)\cong\Z^2$. Then we have a diffeomorphism
	\[ W\cong (S^2\ttimes D^3)\cup_\phi(S^2\ttimes D^3), \]
	see \cite[Section 1]{Ba65}.	The diffeomorphism $\phi$ can explicitly be constructed as follows: Let $A\colon[-1,1]\to SO(4)$ be the map
	\[ A_t= \begin{pmatrix}
		1 & 0 & 0 & 0 \\
		0 & -\sin(\tfrac{\pi}{2} t) & 0 & \cos(\tfrac{\pi}{2} t)\\
		0 &  0 & 1 & 0 \\
		0 & -\cos(\tfrac{\pi}{2} t) & 0 & -\sin(\tfrac{\pi}{2} t)
	\end{pmatrix} \]
	so that $A_{-1}=I_4$ and $A_1$ is complex conjugation $c$ in $\C^2\cong\R^4$. We reparametrize $A$ so that it is constant near $t=\pm1$ and define $\phi$ as the diffeomorphism on $I\times S^3$ given by
	\[ \phi(t,v)=(t,A_t(v)). \]
	Since $c$ on $S^3$ is a bundle map of the Hopf fibration, $\phi$ descends to a diffeomorphism of $\C P^2\#(-\C P^2)$. Since $c$ induces a reflection along a circle in $S^2$, the map $\phi$ induces the map $(x,y)\mapsto (x,-y)$ on homology as required.
	
	Now let $f,h\colon[-1,1]\to[0,\infty)$ be smooth functions and consider the metric
	\[ g_{f,h}=dt^2+f(t)^2\mathcal{V}^*ds_1^2+h(t)^2 H^*ds_2^2 \]
	on $I\times S^3$, where $\mathcal{V}$ it the projection in $TS^3$ onto the vertical part of the Riemannian submersion $H\colon (S^3,ds_3^2)\to (S^2,ds_2^2)$. Similarly as in \ref{EQ:dw_boundary_0} and \ref{EQ:dw_boundary_t0} we require the boundary conditions
	\begin{enumerate}
		\item $f$ is odd at $t=-1$ with $f'(-1)=1$ and $h$ is even at $t=-1$ with $h(-1)>0$,
		\item $f$ is odd at $t=1$ with $f'(1)=-1$ and $h$ is even at $t=1$ with $h(1)>0$,
	\end{enumerate}
	and $f$ and $h$ are strictly positive on $(-1,1)\times S^3$. By \cite{Bu19} this is sufficient for $g_{f,h}$ to define a smooth metric on $S^2\ttimes S^2$. Moreover, if we set
	\[ g_f=dt^2+f(t)^2ds_1^2 \]
	on $I\times S^1$, then these boundary conditions ensure that the metric $g_f$ defines a smooth metric on $S^2$ and, if $h$ is constant, we obtain a Riemannian submersion $(S^2\ttimes S^2,g_{f,h})\to (S^2,h(0)^2ds_2^2)$ with fibres isometric to $(S^2,g_f)$. In particular, we obtain such a metric for $f_0(t)=\tfrac{2}{\pi}\cos(\tfrac{\pi}{2} t)$ and $h_0(t)=\frac{2}{\pi}$.
	
	\begin{lemma}\label{L:pb_path}
		The metric $g_{f_0,h_0}$ has positive Ricci curvature and lies in the same path component of the space of Ricci-positive metrics on $S^2\ttimes S^2$ as the metric $\phi^*g_{f_0,h_0}$.
	\end{lemma}
	\begin{proof}
		The Ricci curvatures of the metric $g_{f,h}$ are given as follows:
		\begin{align*}
			\Ric(\partial_t,\partial_t)&=-\frac{f''}{f}-2\frac{h''}{h},\\
			\Ric(\tfrac{V}{f},\tfrac{V}{f})&=-\frac{f''}{f}-2\frac{f'h'}{fh}+2\frac{f^2}{h^4},\\
			\Ric(\tfrac{X}{h},\tfrac{X}{h})&=-\frac{h''}{h}+\frac{1-{h'}^2}{h^2}-\frac{f'h'}{fh}+3\frac{1}{h^2}-2\frac{f^2}{h^4},\\
			\Ric(\partial_t,X)&=\Ric(\partial_t,V)=\Ric(X,V)=0.
		\end{align*}
		This can for example be obtained from the formulas in Subsection \ref{SS:Proj_bdl} in the case $d=n=2$. A calculation now shows that the metric $g_{f_0,h_0}$ has positive Ricci curvature.
		
		At $t=0$ we have $h_0(t)=f_0(t)$ and $h_0'(t)=f_0'(t)$. Since the Ricci curvatures are linear in the second derivative, replacing $h_0$ by the convex combination $sh_0+(1-s)f_0$ around $t=0$ yields a metric that has positive Ricci curvature at $t=0$ for all $s\in[0,1]$, and therefore also for all $t$ in a neighbourhood of $t=0$. By \cite{BH22} we can smoothly deform $h_0$ into a function $h_1$ that coincides with $f_0$ on $[-\varepsilon,\varepsilon]$ for some $\varepsilon>0$, with $h_0$ outside $(-\varepsilon',\varepsilon')$ for some $\varepsilon'>\varepsilon$ and such that the Ricci curvatures are positive on all of $S^2\ttimes S^2$ along this deformation.

		Hence, the metric $g_{f_0,h_1}$ has positive Ricci curvature and, since $f_0$ and $h_1$ coincide on $[-\varepsilon,\varepsilon]$, it is of the form
		\[ dt^2+f_0(t)^2(\mathcal{V}^*ds_1^2+H^*ds_2^2)=dt^2+f_0(t)^2ds_3^2 \]
		for $t\in[-\varepsilon,\varepsilon]$.
		
		We now reparametrize $A$ to be defined on the interval $[-\varepsilon,\varepsilon]$. The resulting diffeomorphism of $S^2\ttimes S^2$, which we again denote by $\phi$, is then isotopic to $\phi$. Since for each $t$, the map $A_t$ is an isometry of $(S^3,ds_3^2)$, we have that $\phi^*g_{f_0,h_1}$ coincides with $g_{f_0,h_1}$ on each piece where $A_t$ is constant. Thus, we define the homotopy $A_t^s$ for $s\in[0,1]$ by
		\[ A_t^s=A_{-\varepsilon+s(t+\varepsilon)} \]
		and reparametrize it so that $A_t^s$ is constant near $t=\pm\varepsilon$ for all $s\in[0,1]$. Then we have $A_t^0=A_{-\varepsilon}=\mathrm{I}_4$ and $A_t^1=A_t$. We denote the resulting diffeomorphism of $[-\varepsilon,\varepsilon]\times S^3$ by $\phi_s$ (and note that $\phi_s$ does need not extend to all of $S^2\ttimes S^2$ when $s\in(0,1)$). Now define the metric $g_s$ on $S^2\ttimes S^2$ by
		\[ g_s=\begin{cases}
			\phi_s^*g_{f_0,h_1},\quad & t\in[-\varepsilon,\varepsilon],\\
			g_{f_0,h_1},\quad & \text{else.}
		\end{cases} \]
		Since $A_t^s$ is constant near $t=\pm\varepsilon$, the metric $g_s$ is smooth. Further, since it is the pull-back of a metric of positive Ricci curvature, it has positive Ricci curvature for all $s\in[0,1]$. Thus, we have a smooth deformation between $g_{f_0,h_1}$ and $\phi^*g_{f_0,h_1}$ through metrics of positive Ricci curvature. Finally, we can now deform $\phi^*g_{f_0,h_1}$ into $\phi^*g_{f_0,h_0}$ to finish the proof.
	\end{proof}
	
	\begin{proof}[Proof of Theorem \ref{T:Wu}]
		Since $h_0$ is constant, we have a Riemannian submersion $(S^2\ttimes S^2,g_{f_0,h_0})\to(S^2,h(0)^2ds_2^2)$ with fibres isometric to $(S^2,g_{f_0})=(S^2,\frac{4}{\pi^2}ds_2^2)$. Let $\theta$ denote the principal $SO(3)$-connection of this submersion and let $g_r=g_r^\theta$ be the submersion metric as in Subsection \ref{SS:fibre_bdls} with $\check{g}=h(0)^2ds_2^2$, $\hat{g}$ the induced metric on a hemispheres $D^3\subseteq S^3$ in the round sphere of radius $\frac{2}{\pi}$ and principal connection $\theta$. Then the restriction of $g_1$ to the boundary is given by $g_{f_0,h_0}$
		
		By Proposition \ref{P:Ric_bundles}, the metric $g_r$ has positive Ricci curvature for $r$ sufficiently small. Further, the path $r\mapsto g_r|_{S^2\ttimes S^2}=r^2g_{f_0,\frac{1}{r}h_0}$ has positive Ricci curvature for all $r\in[0,1]$ by the formulas in the proof of Lemma \ref{L:pb_path}. Hence, by Propositions \ref{P:deform}, \ref{P:deform_II>=0} and \ref{P:gluing}, we can glue a cylinder to the boundary of $S^2\ttimes D^3$ and obtain a metric of positive Ricci curvature on this glued space with convex boundary and such that the induced metric on the boundary is given by $g_{f_0,h_0}$.
		
		By Propositions \ref{P:deform} and \ref{P:gluing} and Lemma \ref{L:pb_path}, we can attach another cylinder so that the induced metric on the boundary is given by $\phi^*g_{f_0,h_0}$. Hence, we can apply Theorem \ref{T:NE_core_mtrc} with $N=S^2\ttimes D^3$ and $E=S^2\ttimes S^2$, where the identification of $E$ with $\partial N$ is given via $\phi$, to obtain a core metric on $W$.
	\end{proof}
	
	\begin{remark}\label{R:Wm}
		If we replace $S^2\ttimes D^3$ by the linear $D^3$-bundle $\overline{E}\to\C P^{2i-1}$ obtained from the generalized Hopf fibration $S^{4i-1}\to \C P^{2i-1}$ by the inclusion $SO(2)\hookrightarrow SO(4)$, analogous arguments as in the proof of Theorem \ref{T:Wu} yield a core metric on a closed, simply-connected non-spin manifold $W_i^{4i+1}$ with second homology $H_2(W_i)\cong\Z/2$ (see Lemma \ref{L:Wi_cohom} below). In this case we have a diffeomorphism $\partial\overline{E}\cong\C P^{2i}\#(-\C P^{2i})$ and the path $A_t$ similarly connects the identity and complex conjugation on $S^{4i-1}$. The metric $g_{f,h}$ is then given by scaling the base and fibres of the generalized Hopf fibration $(S^{4i-1},ds_{4i-1}^2)\to(\C P^{2i-1},g_{FS})$, where $g_{FS}$ is the Fubini--Study metric on $\C P^{2i-1}$.
	\end{remark}

	\subsection{Oriented bordism}\label{SS:bordism}
	
	In this subsection we prove Corollary \ref{C:bordism}. We first establish the existence of a suitable generating set for $\bigslant{\Omega_*^{SO}}{Tors}$.
	
	\begin{proposition}\label{P:bordsm_generators}
		The ring $\bigslant{\Omega_*^{SO}}{Tors}$ is generated by a sequence $\{M_j^{4j}\}$, where each manifold $M_j$ is a complex projective space or the total space of a projective bundle over a complex projective space.
	\end{proposition}
	\begin{proof}
		It was shown by Milnor \cite{Mi60}, see also \cite[Chapters 7 and 9]{St68}, \cite[Theorems 2 and 3]{Th95} and \cite[Proof of Corollary C]{GL80a}, that $\bigslant{\Omega_*^{SO}}{Tors}$ is generated by a sequence of manifolds $M_j^{4j}$, where each $M_j$ is a hypersurface of degree 1 in a product $\C P^k\times\C P^\ell$ with $k\leq \ell$, i.e.\ $M_j$ is the submanifold
		\[ M_j=\{ ((u_0:\dots:u_k),(v_0:\dots:v_\ell))\in\C P^k\times\C P^\ell\mid u_0v_0+\dots+u_kv_k=0 \}. \]
		The projection onto the first factor gives $M_j$ the structure of a fibre bundle over $\C P^k$ with fibre $\C P^{\ell-1}$. To see that this bundle has structure group contained in $U(\ell)$, note that it is the projective bundle corresponding to the complex vector bundle $E\to\C P^k$, where
		\[ E=\{ ((u_0:\dots:u_k),(v_0,\dots,v_\ell))\in\C P^k\times\C^{\ell+1}\mid u_0v_0+\dots+u_kv_k=0 \}. \]
	\end{proof}
%
	\begin{proof}[Proof of Corollary \ref{C:bordism}]
		By Proposition \ref{P:bordsm_generators}, the ring $\bigslant{\Omega_*^{SO}}{Tors}$ has a generating set consisting of manifolds that all admit a core metric by Theorem \ref{T:sph_proj_bdl}. Hence, by taking products and connected sums using Theorem \ref{T:core_bdl} and \ref{EQ:core3}, we can represent any class in $\bigslant{\Omega_*^{SO}}{Tors}$ by a manifold admitting a core metric, in particular, a metric of positive Ricci curvature.
	\end{proof}
	
	It remains open whether the same result as in Corollary \ref{C:bordism} holds for the full oriented bordism ring $\Omega_*^{SO}$. However, we have the following result in low dimensions.	
	\begin{proposition}\label{P:bordism_tors}
		Every class in $\Omega_n^{SO}$ with $n\leq 12$ is represented by a connected manifold admitting a Riemannian metric of positive Ricci curvature.
	\end{proposition}
	To prove Proposition \ref{P:bordism_tors} we will use the manifolds $W_i$ introduced in Remark \ref{R:Wm}. For that we first prove the following lemma.	
	\begin{lemma}\label{L:Wi_cohom}
		The cohomology groups of the manifold $W_i$ with coefficients in $\Z/2$ are given by
		\[ H^j(W_i;\Z/2)\cong\begin{cases}
			\Z/2,\quad & 2\leq j\leq 4i-1\text{ or }j=0,4i+1,\\
			0,\quad & \text{else}.
		\end{cases} \]
		If $a\in H^2(W_i;\Z/2)\cong\Z/2$ denotes a generator, then $a^j$ is non-trivial for all $j\in\{1,\dots,2i-1\}$. Further, the Stiefel--Whitney classes of $W_i$ are given as follows:
		\[ w_{2j}(W_i)={2i+1\choose j}a^j,\quad w_{2j+1}(W_i)=j{2i+1\choose j}(a^{2i-j})^*,\]
		where $(a^{2i-j})^*\in H^{2j+1}(W_i;\Z/2)$ is the Poincaré dual of $a^{2i-j}$.
	\end{lemma}
	
	In particular it follows that each $W_i$ represents a non-trivial class in $\Omega_{4i+1}^{SO}$ (and, more generally, also in the unoriented bordism group $\Omega_{4i+1}^{O}$), since the Stiefel--Whitney number $w_3w_{2}^{2i-1}(W_i)$ is non-trivial.
	
	\begin{proof}
		Let $\overline{E}\xrightarrow{\pi}\C P^{2i-1}$ be the linear $D^3$-bundle over $\C P^{2i-1}$ obtained from the generalized Hopf fibration $S^{4i-1}\to\C P^{2i-1}$ via the inclusion $SO(2)\hookrightarrow SO(4)$. Then $E=\partial\overline{E}\xrightarrow{\pi|_E}\C P^{2i-1}$ is a linear $S^2$-bundle and we have an inclusion $S^{4i-1}\hookrightarrow E$ that divides $\overline{E}$ into two linear $D^2$-bundles over $\C P^{2i-1}$. These bundles have the same Euler class as the generalized Hopf fibration, i.e.\ it is given by a generator of $H^2(\C P^{2i-1})\cong\Z$, which equals the Euler class of the normal bundle of $\C P^{2i-1}\subseteq\C P^{2i}$. Hence, by identifying $\C P^{2i}\setminus{D^{4i}}^\circ$ with a tubular neighbourhood of $\C P^{2i-1}$ in $\C P^{2i}$, we obtain an identification
		\[ E\cong\C P^{2i}\#(-\C P^{2i}). \]
		We denote by $i_E\colon\C P^{2i}\#(-\C P^{2i})\cong E\hookrightarrow \overline{E}$ the inclusion. Hence, $W_i$ fits into the following pushout diagram:
		\[
		\begin{tikzcd}
			\C P^{2i}\#(-\C P^{2i})\arrow[hookrightarrow]{r}{i_E\circ\phi}\arrow[hookrightarrow]{d}{i_E} & \overline{E}\arrow[hookrightarrow]{d}\\
			\overline{E}\arrow[hookrightarrow]{r} & W_i
		\end{tikzcd}
		\]
		where $\phi\colon \C P^{2i}\#(-\C P^{2i})\to \C P^{2i}\#(-\C P^{2i})$ is a diffeomorphism inducing the map $(x,y)\mapsto (x,-y)$ on $H_2(\C P^{2i}\#(-\C P^{2i}))\cong H_2(\C P^{2i-1})\oplus H_2(\C P^{2i-1})\cong\Z^2$.
		
		Now let $R$ be an arbitrary commutative ring. We have $\overline{E}\simeq \C P^{2i-1}$, and, according to the identification described above, the inclusion $\C P^{2i}\setminus D^{4i}\hookrightarrow \C P^{2i}\#(-\C P^{2i})\hookrightarrow \overline{E}$ into the first resp.\ second summand maps $\C P^{2i-1}\subseteq\C P^{2i}\setminus D^{4i}$ into $\overline{E}$ as a section. In particular, it induces the identity map on $H_2(\C P^{2i-1};R)$. Hence, by the universal coefficient theorem, we can choose generators $\bar{b}_1,\bar{b}_2\in H^2(\overline{E};R)\cong H^2(\C P^{2i-1};R)\cong R$ for each copy of $\overline{E}$ and generators $b_1,b_2\in H^2(\C P^{2i}\#(-\C P^{2i});R)$ of the first and second summand of $H^2(\C P^{2i}\#(-\C P^{2i});R)$, respectively, so that the map $i_E^*$ is given by
		\[ i_E^*(\bar{b}_j)=b_1+b_2 \]
		for $j=1,2$.
		
		From the Mayer--Vietoris sequence we obtain the following exact sequence:
		\begin{equation}\label{EQ:MV}
			0\longrightarrow H^{2j}(W_i;R)\longrightarrow H^{2j}(\overline{E};R)\oplus H^{2j}(\overline{E};R)\longrightarrow H^{2j}(E;R)\longrightarrow H^{2j+1}(W_i;R)\longrightarrow 0,
		\end{equation}
		where the map $H^{2j}(\overline{E};R)\oplus H^{2j}(\overline{E};R)\to H^{2j}(E;R)$ is given by
		\[ \bar{b}_1^j\mapsto (b_1+b_2)^j=b_1^j+b_2^j,\quad \bar{b}_2^j\mapsto -(b_1-b_2)^j=-b_1^j-(-b_2)^j. \]
		In particular, for $R=\Z/2$ we have $\bar{b}_1^j,\bar{b}_2^j\mapsto b_1^j+b_2^j$, so that $H^{2j}(W_i;\Z/2)\cong H^{2j+1}(W_i;R)\cong\Z/2$ for $1\leq j\leq 2i-1$ and $\bar{b}_1+\bar{b}_2$ defines a generator $a$ of $H^2(W_i;\Z/2)$ such that $a^j$ is a generator of $H^{2j}(W;\Z/2)$ for all $j\leq 2i-1$.
		
		For the Stiefel--Whitney classes first note that from the decomposition $T\overline{E}\cong \pi^*\xi\oplus\pi^* T\C P^{2i-1}$, where $\xi$ is the vector bundle corresponding to $\pi$, and from the fact that for the generalized Hopf fibration only $w_0$ and $w_2$ are non-zero, we obtain
		\begin{align*}
			w_{2j}(\overline{E})&=\sum_{k=0}^{2j}w_k(\pi)\smile w_{2j-k}(\C P^{2i-1})\\
			&=w_{2j}(\C P^{2i-1})+w_{2j-2}(\C P^{2i-1})\smile b\\
			&={2i\choose j}b^j+{2i\choose j-1}b^j\\
			&={2i+1\choose j}b^j.
		\end{align*}
		Hence, since the map $H^{2j}(W_i;\Z/2)\to H^{2j}(\overline{E};\Z/2)\oplus H^{2j}(\overline{E};\Z/2)$ is injective, we obtain that
		\[w_{2j}(W_i)={2i+1\choose j}a^j.\]
		For the odd-degree Stiefel--Whitney classes we use the Wu formula (see e.g.\ \cite[Problem 8A]{MS74}) to obtain $w_{2j+1}(W_i)=Sq^1(w_{2j}(W_i))$. To determine the Steenrod square $Sq^1$ we use the naturality of the sequence \eqref{EQ:MV} to deduce that
		\[Sq^1\colon \Z/2\cong H^{2j}(W_i;\Z/2)\to H^{2j+1}(W_i;\Z/2)\cong \Z/2\]
		for $2\leq j\leq 2i-1$ is trivial if and only if $j$ is even. Indeed, the map $H^{2j}(\overline{E};\Z/4)\oplus H^{2j}(\overline{E};\Z/4)\to H^{2j}(E;\Z/4)$ has kernel generated by $\bar{b}_1^j-\bar{b}_2^j$ when $j$ is even and by $2\bar{b}_1^j-2\bar{b}_2^j$ when $j$ is odd. It follows that in the first case the coefficient homomorphism $H^{2j}(W_i;\Z/2)\to H^{2j}(W_i;\Z/4)$ is given by $\Z/2\to\Z/4,\,z\mapsto 2z$ in the first case and by $\Z/2\to\Z/2,\,z\mapsto z$ in the second. Hence, the Bockstein homomorphism $\beta\colon H^2(W_i;\Z/2)\to H^3(W_i;\Z/2)$, which equals the Steenrod square $Sq^1$, is trivial if and only if $j$ is even. Hence, $w_{2j+1}(W_i)$ is trivial if and only if ${2i+1\choose j}$ or $j$ are even.
	\end{proof}
	
	\begin{proof}[Proof of Proposition \ref{P:bordism_tors}]
		We first note that the only groups $\Omega_n^{SO}$ with $n\leq 12$ that contain torsion are in dimensions $n=5,9,10,11$, which can be seen as follows: By \cite[Theorem 4]{Wa60}, the ring $\Omega_*^{SO}$ is generated by a set of torsion-free classes in dimensions $4k$, $k\in\N$, and torsion classes in odd dimensions, where in dimensions $n\leq 12$ we only have a torsion generator in dimensions $5$, $9$ and $11$. Thus, by taking products we obtain the following:
		\begin{align*}
			\Omega_5^{SO}\cong\Z/2,\quad \Omega_9^{SO}\cong(\Z/2)^2,\quad \Omega_{10}^{SO}\cong\Z/2,\quad\Omega_{11}^{SO}\cong\Z/2
		\end{align*}
		and in all other dimensions $n\leq 12$ the group $\Omega_n^{SO}$ is torsion-free (cf.\ also \cite[p.\ 203]{MS74}). The group $\Omega_5^{SO}$ is generated by the Wu manifold $W_1$ and the group $\Omega_{10}^{SO}$ is generated by its square $W_1\times W_1$, which both admit a Riemannian metric of positive Ricci curvature. Further, the group $\Omega_9^{SO}$ is generated by $W_1\times \C P^2$ and $W_2$, which can be seen by considering the Stiefel--Whitney numbers $w_3 w_2^3$ and $w_7w_2$. Indeed, by Lemma \ref{L:Wi_cohom}, we have
		\[ w(W_1)=1+a+a^*,\quad w(W_2)=1+a+(a^3)^*, \]
		so that
		\[ w(W_1\times\C P^2)=(1+a+a^*)(1+b)^3=1+(a+b)+a^*+(ab+b^2)+a^*b+ab^2+a^*b^2. \]
		In particular, we have
		\[ w_3w_2^3(W_2)=1,\quad w_7w_2(W_2)=0,\quad w_3w_2^3(W_1\times\C P^2)=1,\quad w_7w_2(W_1\times\C P^2)=1. \]
		Hence, the manifolds $W_1\times\C P^2$ and $W_2$ generate $\Omega_9^{SO}$, so by Theorems \ref{T:core_bdl} and \ref{T:Wu} and Remark \ref{R:Wm} every class in $\Omega_9^{SO}$ is represented by a connected manifold admitting a Riemannian metric of positive Ricci curvature.
		
		Finally, the non-trivial class in $\Omega_{11}^{SO}$ is represented by the Dold manifold $P(3,4)$, see \cite[Section 3]{Wa60}, which is the quotient of $S^3\times \C P^4$ by the $\Z/2$-action $(x,z)\mapsto(-x,\bar{z})$. The product metric of the round metric and the Fubini--Study metric is invariant under this action, hence $P(3,4)$ admits a metric of positive Ricci curvature.
	\end{proof}
	
	
	Similarly, we can ask which classes in the spin bordism ring can be represented by Riemannian manifolds of positive Ricci curvature. An obstruction here is the $\hat{A}$-genus and we can ask the following question:
	\begin{question}
		Can every class in $\bigslant{\Omega_*^{Spin}}{Tors}$ with vanishing $\hat{A}$-genus (or, more generally, every class in $\Omega_*^{Spin}$ with vanishing $\alpha$-invariant) be represented by a connected manifold admitting a Riemannian metric of positive Ricci curvature?
	\end{question}
	The rational spin bordism ring $\Omega_*^{Spin}\otimes\Q$ has a basis given by
	\[\{[K3]\}\cup\{[\Quat P^i]\}_{i\geq 2}. \]
	Hence, by taking connected sums of products of quaternionic projective spaces, we obtain the following partial answer:
	\begin{proposition}\label{P:spin_bord}
		For every $k\in\N$ there exists a subspace of $\Omega_{4k}^{Spin}$ of rank $p(k)-p(k-1)$ in which each element is represented by a connected manifold admitting a Riemannian metric of positive Ricci curvature.
	\end{proposition}
	Here $p$ denotes the \emph{partition function}, i.e.\ $p(k)$ is the number of possible partitions of $k$ into non-negative integers.

	\bibliographystyle{plainurl}
	\bibliography{References}

\end{document}